\documentclass[12pt]{amsart}
\usepackage{amssymb}
\usepackage{amsfonts}
\usepackage{latexsym}
\usepackage[all]{xy}
\usepackage{amscd}
\usepackage{comment}
\usepackage{fullpage}

\newtheorem{theorem}{Theorem}[section]
\newtheorem{lemma}[theorem]{Lemma}
\newtheorem{proposition}[theorem]{Proposition}
\newtheorem{corollary}[theorem]{Corollary}
\theoremstyle{definition}
\newtheorem{definition}[theorem]{Definition}

\newtheorem{remark}[theorem]{Remark}
\newtheorem{note}[theorem]{Note}

\newcommand{\C}{\mathcal{C}}
\newcommand{\D}{\mathcal{D}}
\newcommand{\A}{\mathcal{A}}
\newcommand{\B}{\mathcal{B}}
\newcommand{\E}{\mathcal{E}}

\newcommand{\Z}{\mathcal{Z}}

\renewcommand{\S}{\mathcal{S}}

\newcommand{\id}{\text{id}}

\newcommand{\ot}{\otimes}
\newcommand{\be}{\mathbf{1}}

\newcommand{\lexp}[2]{{\vphantom{#2}}^{#1}{#2}}

\DeclareMathOperator{\Ind}{Ind}

\DeclareMathOperator{\Hom}{Hom}

\DeclareMathOperator{\Rep}{Rep}

\DeclareMathOperator{\Ker}{Ker }
\DeclareMathOperator{\op}{op} 
\DeclareMathOperator{\Irr}{Irr}

\renewcommand{\Vec}{\operatorname{Vec}}
\renewcommand{\dim}{\operatorname{dim }}
\renewcommand{\deg}{\operatorname{deg }}

\begin{document}

\title{Fusion subcategories of representation categories
of twisted quantum doubles of finite groups}

\author{Deepak Naidu}
\address{Department of Mathematics, Texas A\&M University, College Station,
TX 77843, USA}
\email{dnaidu@math.tamu.edu}

\author{Dmitri Nikshych}
\address{Department of Mathematics and Statistics, University of New Hampshire,
Durham, NH 03824, USA}
\email{nikshych@math.unh.edu}

\author{Sarah Witherspoon}
\address{Department of Mathematics, Texas A\&M University, College Station,
TX 77843, USA}
\email{sjw@math.tamu.edu}

\date{\today}

\begin{abstract}
We describe all fusion subcategories of the representation category $\Rep(D^{\omega}(G))$
of a twisted quantum double 
$D^\omega(G)$, where $G$ is a finite group and $\omega$ 
is a $3$-cocycle on $G$. 
In view of the fact that every
group-theoretical braided fusion category can be embedded
into some $\Rep(D^{\omega}(G))$, this
gives a complete description of all group-theoretical 
braided fusion categories. We describe the lattice and give formulas
for some invariants of the fusion subcategories of $\Rep(D^\omega(G))$. 
We also give a characterization of group-theoretical braided fusion
categories as equivariantizations of pointed categories.
\end{abstract}

\maketitle

%%%%%%%%%%%%%%%%%%%%%%%%%%%%%%%%%%%%%%%%%%%%%%%%%%%%%%%%%%%%%%%%%%%
\begin{section}
{Introduction}

Let $G$ be a finite group and $\omega$ be a $3$-cocycle on $G$.
In \cite{DPR1, DPR2} Dijkgraaf, 
Pasquier, and Roche introduced a quasi-triangular quasi-Hopf algebra 
$D^\omega(G)$. When $\omega =1$ this quasi-Hopf algebra coincides with 
the Drinfeld double $D(G)$ of $G$ and so $D^\omega(G)$ is often called
a {\em twisted quantum double} of $G$. It is well known that the 
representation category $\Rep(D^\omega(G))$ of $D^\omega(G)$ 
is a modular category  \cite{BK} 
and is braided equivalent to the center $\Z(\Vec_G^\omega)$ 
of the tensor category $\Vec_G^\omega$ of finite-dimensional 
$G$-graded vector spaces with associativity constraint
defined using $\omega$ \cite{Mj}.
%The category $\Vec_G^\omega$ is a typical
%example of a {\em pointed} fusion category, i.e., a fusion category in 
%which every simple object is invertible.

The principal goal of this paper is to give a
complete description of fusion subcategories of
$\Rep(D^\omega(G))$ and the lattice formed by them.
Our description of the lattice may shed more light on
the structure of $\Rep(D^\omega(G))$. For instance, the group of braided autoequivalences
of $\Rep(D^\omega(G))$ acts on this lattice, and so one can, in
principle, derive some information about this group from our description.

Our results also have consequences more generally for a
{\em group-theoretical} braided fusion category, that is a braided category  
$\C$ {\em dual} to $\Vec_G^\omega$ for some $G$ and $\omega$ 
(see Section \ref{fusion cats.} or \cite{O2}). 
%Suppose that $\C$ is, in addition,
%braided.
Since $\C$ is braided, it can be realized as a fusion subcategory of
$\Rep(D^\omega(G)) \cong \Z(\Vec_G^\omega)$. 
This means that our description of all  fusion 
subcategories of $\Rep(D^\omega(G))$, for
all finite groups $G$ and $3$-cocycles $\omega$, is equivalent to a 
description of all  group-theoretical braided fusion categories. 
In particular, our results give a description of all {\em modular}
\cite{BK} group-theoretical fusion categories.

Group-theoretical fusion categories are of interest as examples 
of fusion categories that can be described explicitly in terms of
finite groups and their cohomology \cite{O1}.
They are also more general than one might at first expect:
All semisimple quasi-Hopf algebras of prime power dimension have
group-theoretical representation categories, a consequence of
a more general result on nilpotent fusion categories \cite{DGNO}.
For some time it was unknown whether there are any complex semisimple
Hopf algebras having non-group-theoretical representation categories;
the first known example was announced recently by the second author \cite{Nk}.

%%%%%%%%%%%%%%%%%%%%%%%%%%%%%%%%%%%%%%%%%%%%%%%%%%%%%%%%%%%%%%%%%%%%%%%%%%%%%
\subsection{Main results}

Let $G$ be a finite group,
and $k$ an algebraically closed
field  of characteristic zero. Unless otherwise stated
all cocycles appearing in this 
work will have coefficients in the trivial module $k^\times$.
All categories
are assumed to be $k$-linear and semisimple with finite-dimensional
$\Hom$-spaces and finitely many isomorphism classes of simple objects.
All functors are assumed to be additive and $k$-linear.

\begin{theorem}
\label{thm1}
Fusion subcategories of the representation category
of the Drinfeld double $D(G)$ are in bijection with triples $(K,H,B)$, where $K,H$ are
normal subgroups of $G$ centralizing each other and $B: K\times H\to k^\times$
is a $G$-invariant bicharacter.
\end{theorem}

Theorem~\ref{thm1} gives a simple description of all fusion
subcategories for the untwisted double $D(G)$.  
Now let $\omega \in Z^3(G,\, k^\times)$ be a $3$-cocycle on $G$ .
In the twisted
($\omega \neq 1$) case the notion of a $G$-invariant
bicharacter must be twisted (see Definition~\ref{inv-bichar}).

\begin{theorem}
\label{thm2}
Fusion subcategories of the representation category
of the twisted double $D^\omega(G)$ are in bijection with triples $(K,H,B)$ 
where $K,H$ are normal subgroups
of $G$ centralizing each other and $B: K\times H \rightarrow k^{\times}$
is a $G$-invariant $\omega$-bicharacter.
\end{theorem}

As noted earlier, the above result is equivalent to a 
description of all group-theoretical braided fusion categories. 

Let $(K,H,B)$ be a triple as described in the above theorem and let
$\S(K,H,B)$ denote the corresponding fusion subcategory of $\Rep(D^{\omega}(G))$.
The next result gives a criterion for the fusion subcategory 
${\mathcal S}(K,H,B) \subseteq \Rep(D^{\omega}(G))$
to be nondegenerate or prime. 
Recall \cite{M1} that a nondegenerate braided fusion category  is called
{\em prime} if it has no proper nontrivial nondegenerate subcategories.
%The result below also gives a criterion for $\Rep(D^{\omega}(G))$ to be prime.

\begin{theorem}
\label{thm3}
\begin{enumerate}
\item[(i)] The fusion subcategory ${\mathcal S}(K,H,B) \subseteq \Rep(D^{\omega}(G))$
is nondegenerate if and only if $HK=G$ and the symmetric bicharacter
$BB^{\op}|_{(K\cap H)\times (K\cap H)}$ is nondegenerate.
\item[(ii)] $\Rep(D^{\omega}(G))$ is prime if and only if there is no triple
$(K,H,B)$, where $K$ and $H$ are proper normal subgroups of $G$
that centralize each other, $HK=G$, and $B$
is a $G$-invariant $\omega$-bicharacter on $K\times H$ such that 
$BB^{\op}|_{(K\cap H)\times (K\cap H)}$ is nondegenerate.
\end{enumerate}
\end{theorem}

We note that, in our situation, being nondegenerate is equivalent to being modular 
(see Section \ref{prelim: centralizers}).
Therefore, the above result combined with Theorem \ref{thm2} gives a
description of all modular group-theoretical fusion categories.

%The following result gives an alternate description of all nondegenerate group-theoretical 
%braided fusion categories in terms of Tannakian subcategories and their corresponding 
%fiber categories \cite{DGNO}.

%\begin{theorem}
%\label{thm3}
%Let $\C$ be a braided fusion category. Then $\C$ is group-theoretical if and only if
%it contains a Tannakian subcategory $\Rep(G)$ such that the corresponding fiber category
%is pointed.
%\end{theorem}

%%%%%%%%%%%%%%%%%%%%%%%%%%%%%%%%%%%%%%%%%%%%%%%%%%%%%%%%%%%%%%%%%%%%%%%%%%%%%
\subsection{Organization of the paper}
Section 2 contains necessary preliminary information
about fusion categories and module categories.
We also recall definitions and
results from \cite{M1, DGNO} concerning centralizers in braided
fusion categories.

Section 3 (respectively, Section 5) is devoted to a description
of fusion subcategories of  the representation
category of the Drinfeld double (respectively, twisted double)
of a finite group. We chose to treat the untwisted
and twisted cases separately because the twisted case involves rather technical 
cohomological computations which may obscure the main results. 
We note that when $\omega =1$ the results of Section~5 reduce 
to those of Section~3.

Sections 3, 5, and 6 contain proofs of some of
our main results stated above. Theorems~\ref{thm1}, \ref{thm2} 
and Theorem~\ref{thm3} correspond to
Theorems~\ref{bijection}, \ref{theorem:twisted-main}, and
Theorem \ref{thm:nondegenerate}, respectively.

Section 4 (respectively, Section 6) is devoted to a description
of the lattice and some invariants of fusion subcategories of the representation
category of the Drinfeld double (respectively, twisted double)
of a finite group. 

In Section~7 we prove that a braided fusion category is group-theoretical
if and only if it is equivalent to an equivariantization
of a pointed fusion category with respect to an action of a group $G$.

%%%%%%%%%%%%%%%%%%%%%%%%%%%%%%%%%%%%%%%%%%%%%%%%%%%%%%%%%%%%%%%%%%%%%%%%%%%%%
\subsection{Acknowledgments}
The second author was partially supported 
by NSA grant H98230-07-1-0081 and NSF grant DMS-0800545. He is grateful
to Vladimir Drinfeld, Pavel Etingof, and Viktor Ostrik for helpful discussions.
The third author was partially supported by NSA grant H98230-07-1-0038
and NSF grant DMS-0800832.
\end{section}
%%%%%%%%%%%%%%%%%%%%%%%%%%%%%%%%%%%%%%%%%%%%%%%%%%%%%%%%%%%%%%%%%%%%%

%%%%%%%%%%%%%%%%%%%%%%%%%%%%%%%%%%%%%%%%%%%%%%%%%%%%%%%%%%%%%%%%%%%%%
\begin{section}
{Preliminaries}

%%%%%%%%%%%%%%%%%%%%%%%%%%%%%%%%%%%%%%%%%%%%%%%%%%%%%%%%%%%%%%%%%%
\begin{subsection}{Group-theoretical fusion categories}
\label{fusion cats.}

A {\em fusion category} over $k$ is a
$k$-linear semisimple rigid tensor category with finitely many
isomorphism classes of simple objects and finite dimensional Hom-spaces
such that the neutral object is simple \cite{ENO}.

A fusion category is said to be {\em pointed} if all its simple
objects are invertible. A typical example of a pointed category
is $\Vec_G^{\omega}$, the category of finite dimensional vector
spaces over $k$ graded by the finite group $G$. The morphisms in this category
are linear transformations that respect the grading and the associativity
constraint is given by the normalized $3$-cocycle $\omega$ on $G$.

Consider the fusion category $\Vec_G^{\omega}$. Let
$H$ be a subgroup of $G$ such that $\omega|_{H \times H \times H}$
is cohomologically trivial. Let $\psi$ be a $2$-cochain in $C^2(H,
\, k^\times)$ satisfying $\omega|_{H \times H \times H} = \delta^2\psi$.
The twisted group algebra $k^{\psi}[H]$ is an associative unital
algebra in $\Vec_G^{\omega}$. Define $\C = \C(G, \, \omega, \, H,
\, \psi)$ to be the category of $k^{\psi}[H]$-bimodules in
$\Vec_G^{\omega}$. Then $\C$ is a fusion category with tensor
product $\ot_{k^{\psi}[H]}$ and unit object $k^{\psi}[H]$.
Categories of the form $\C(G, \, \omega, \, H, \, \psi)$
are known as {\em group-theoretical} \cite[Defn.~8.40]{ENO}, \cite{O2}.
It is known that a fusion category
$\C$ is group-theoretical if and only if it is dual
to a pointed category with respect to an indecomposable module category. 
(See \cite{O1} and \cite{M3} for notions of module category and duality.)

\end{subsection}
%%%%%%%%%%%%%%%%%%%%%%%%%%%%%%%%%%%%%%%%%%%%%%%%%%%%%%%%%%%%%%%%%%

%%%%%%%%%%%%%%%%%%%%%%%%%%%%%%%%%%%%%%%%%%%%%%%%%%%%%%%%%%%%%%%%%%
\begin{subsection}{Adjoint categories and central series of fusion categories}

The following definitions were introduced in \cite{GN, ENO}.

Let $\C$ be a fusion category. 
The {\em adjoint category} $\C_{ad}$ of $\C$
is the full fusion subcategory of $\C$ generated by all subobjects of $X\ot X^*$, where $X$ runs
through simple objects of $\C$.  For example, let $G$ be a finite group and let $\C=\Rep(G)$ be 
the representation category of $G$. Then $\C_{ad} = \Rep(G/Z(G))$, where $Z(G)$ is the center of $G$.

Let $\C^{(0)}= \C,\, \C^{(1)} = \C_{ad},$
and $\C^{(n)} = (\C^{(n-1)})_{ad}$ for every integer $n\geq 1$.

The non-increasing sequence of fusion subcategories of $\C$
\begin{equation*}
\C= \C^{(0)} \supseteq  \C^{(1)} \supseteq \cdots
\supseteq \C^{(n)} \supseteq \cdots
\end{equation*}
is called the {\em upper central series} of $\C$.
This definition generalizes the classical one, as we show next:
Let $G$ be a finite group with identity element $e$, and $\C=\Rep(G)$. Let
\[
\{e\} = C^0(G) \subseteq C^{1}(G) \subseteq
\cdots  \subseteq   C^{n}(G) \subseteq \cdots
\]
be the upper central series of $G$; i.e.,
$C^0(G):=\{e\}, C^{1}(G):= Z(G)$ (the center of $G$) and for $n\geq 1$
the subgroup  $C^n(G)$ is defined by $C^n(G)/C^{n-1}(G) = Z(G/C^{n-1}(G))$.
Then $\C^{(n)} = \Rep(G/C^n(G))$.

\begin{comment}
A fusion category $\C$ is said to be {\em nilpotent} if
its upper central series converges to $\Vec$ (the fusion subcategory
generated by the unit object); i.e.,
$\C^{(n)} =\Vec$ for some $n$. The smallest such $n$
is called the  {\em nilpotency class} of $\C$. We note that for a finite
group $G$, its representation category $\Rep(G)$ is nilpotent if and only
if $G$ is nilpotent.
\end{comment}

Let $\C$ be a fusion category such that its Grothendieck ring $K_0(\C)$ is commutative
(e.g., $\C$ is braided).
Let $\mathcal{\D}$ be a fusion subcategory of $\C$.
The {\em commutator} $\D^{co}$ of $\D$ is defined
to be the fusion subcategory
of $\C$ generated by all simple objects $Y$ of $\C$ such that $Y\otimes Y^*\in \D$.
An example:
Let $G$ be a finite group and let $\C=\Rep(G)$.
Any fusion subcategory  of $\C$ is of the form
$\Rep(G/N)$ for some normal subgroup $N$ of $G$.
The simple objects of the
category $\Rep(G/N)^{co}$  are irreducible 
representations $Y$ of $G$ for which $ Y \otimes Y^*$  restricts
to the trivial representation of $N$.

Let $\C_{(0)}= \Vec,\, \C_{(1)} = \C_{pt}$
(the maximal pointed subcategory of $\C$)
and $\C_{(n)} = (\C_{(n-1)})^{co}$ for every integer $n\geq 1$.
The non-decreasing sequence of fusion subcategories of $\C$
\begin{equation}
\Vec= \C_{(0)} \subseteq  \C_{(1)} \subseteq \cdots \subseteq
\C_{(n)} \subseteq \cdots
\end{equation}
is called the {\em lower central series} of $\C$.
This definition generalizes the classical one:
Let $G$ be a finite group and  $\C=\Rep(G)$. Let
\[
G = C_0(G) \supseteq C_{1}(G) \supseteq   \cdots  \supseteq
C_{n}(G) \supseteq \cdots
\]
be the lower central series of $G$; i.e.,
$C_n(G) = [G, C_{n-1}(G)]$ for all $n\geq 1$.
Then $\C_{(n)} = \Rep(G/C_n(G))$.

\end{subsection}
%%%%%%%%%%%%%%%%%%%%%%%%%%%%%%%%%%%%%%%%%%%%%%%%%%%%%%%%%%%%%%%%%%

%%%%%%%%%%%%%%%%%%%%%%%%%%%%%%%%%%%%%%%%%%%%%%%%%%%%%%%%%%%%%%%%%%
\begin{subsection}{Centralizers in braided fusion categories}
\label{prelim: centralizers}

Let $\C$ be a premodular category (i.e., a braided ribbon fusion category) 
with braiding $c$  and twist $\theta$ (see \cite{BK}). By the {\em $S$-matrix}
of $\C$ we mean a square matrix $S:=( S(X,\,Y))$
whose columns and rows are labeled by simple objects of $\C$ and the entry
$S(X,\,Y)$ is the (quantum) trace (defined using $c$ and $\theta$)
of $c_{Y,X}c_{X,Y} : X\ot Y \to X\ot Y$.

Following M\"uger,  we will say that two objects 
$X, Y \in \C$ {\em centralize}  each other if 
$$
c_{Y, X} \circ c_{X, Y} = \id_{X \otimes Y}.
$$ 

Assuming that categorical dimensions in $\C$ are positive,
simple objects $X$ and $Y$ in $\C$ centralize each other
if and only if $S(X, \, Y) = d(X) d(Y)$ \cite[Prop.\ 2.5]{M1}.
All categories considered in this work have positive categorical
dimensions.

Let $\D$ be a full (not necessarily tensor) 
subcategory of $\C$.  In \cite{M1}, M\"uger introduced the notion of
the {\em centralizer} of $\D$ in $\C$ as 
the fusion subcategory  
\[
\D' := \left\{ X \in \C \mid c(Y, \, X) \circ c(X, \, Y)
= \id_{X \otimes Y}, \mbox{ for all } Y \in \D \right\}. 
\]

We define the {\em M\"uger center} of $\C$, denoted $\Z_2(\C)$, to be
the full fusion subcategory $\C'$. 
The category $\C$ is said to be {\em nondegenerate} if $\Z_2(\C) = \Vec$ (the fusion
subcategory generated by the unit object). Such a category $\C$ is called {\em prime}
if it has no non-trivial proper nondegenerate subcategories. It was shown in \cite{M1} that a nondegenerate 
braided fusion category factorizes (in general, non uniquely) into a product of prime ones.

Note that the corresponding $S$-matrix $S$ of $\C$ is invertible, i.e., $\C$ is modular,
if and only if $\C$ is nondegenerate.%and $\dim(\C)\not=0$.

It is also known \cite{M1} that if $\D$
is a fusion subcategory of a modular category $\C$ then $\D''=\D$ and
\begin{equation}
\label{dimension}
\dim(\D) \cdot \dim(\D') = \dim(\C).
\end{equation}

\begin{lemma}
\label{abelian subcat. is fusion}
Let $\A$ and $\B$ be full Abelian subcategories of $\C$ (a priori not 
assumed to be closed under tensor product or taking duals) centralizing each
other and $dim(\A) \cdot dim(\B) = dim(\C)$. Then $\A$ and $\B$ are
fusion subcategories and $\A'=\B$.
\end{lemma}
\begin{proof}
Let $\tilde{\A}$ denote the fusion subcategory generated by $\A$. 
Note that $\B$ centralizes $\tilde{\A}$, i.e., $\B \subseteq (\tilde{\A})'$.
We have
$$
\dim(\C) 
= \dim(\tilde{\A}) \cdot \dim((\tilde{\A})') \\
 \geq \dim(\A) \cdot \dim(\B) \\
= \dim(\C).
$$
Hence, the inequality above is an equality, so 
$\A$ and $\B$ are fusion subcategories and $\A'=\B$.
\end{proof}

\end{subsection}

\end{section}

%%%%%%%%%%%%%%%%%%%%%%%%%%%%%%%%%%%%%%%%%%%%%%%%%%%%%%%%%%%%%%%%%%%%%%%%%%%%%%%%%%%%%%%%%%%%%%%%%

%%%%%%%%%%%%%%%%%%%%%%%%%%%%%%%%%%%%%%%%%%%%%%%%%%%%%%%%%%%%%%%%%%%%%%%%%%%%%%%%%%%%%%%%%%%%%%%%%
\begin{section}
{Fusion subcategories of the quantum double of 
a finite group}

%%%%%%%%%%%%%%%%%%%%%%%%%%%%%%%%%%%%%%%%%%%%%%%%%%%%%%%%%%%%%%%%%%%%%%%%%%
\begin{subsection}
{The quantum double of a finite group}

Let $G$ be a finite group. 
For any $g \in G$, let $K_g$ denote the conjugacy
class of $G$ containing $g$. 
Let $R$ denote a complete set of representatives
of conjugacy classes of $G$. Let $\C$ denote the representation category
$\Rep(D(G))$ of the Drinfeld double of the group $G$:
\begin{equation*}
\C := \Rep(D(G)).
\end{equation*}

\begin{remark}
\label{C = bundles}
\begin{enumerate}
\item [(i)]
It is well known that $\C$ is equivalent, as a braided fusion category,
to the category of $G$-equivariant vector bundles on $G$, where $G$ acts on itself
by conjugation.
\item[(ii)]  The category $\C$ contains the category $\Rep(G)$ of representations of $G$
as a  fusion subcategory. Namely, $\Rep(G)$ is identified with the category
of $G$-equivariant bundles supported on the identity element of $G$.
\end{enumerate}
\end{remark}

The set of isomorphism classes of simple objects  of $\C$ is identified with 
the following set:
\begin{equation}
\label{Gamma}
\Gamma := \{(a, \, \chi) \mid a \in R \mbox{ and } \chi \mbox{ is an 
irreducible character of } C_G(a) \}.
\end{equation}
The categorical dimension of the simple object $(a, \, \chi)$ is
\begin{equation}
\label{d}
d(a, \, \chi) = |K_a| \, \deg \chi = \frac{|G|}{|C_G(a)|} \, \deg \chi.
\end{equation}

It is well known that $\C$ is a modular category  \cite{AC1}.
Its  $S$-matrix $S$ and twist $\theta$ are given by the following formulas :
%Recall that we take the canonical twist.
%
%
\begin{equation*}
\begin{split}
S((a, \, \chi), \, (b, \, \chi^\prime)) 
&= \frac{|G|}{|C_G(a)||C_G(b)|} \sum_{g \in G(a, \, b)}
\overline{\chi}(gbg^{-1}) \, \overline{\chi}^\prime(g^{-1}ag),\\
\theta(a, \, \chi) 
& = \frac{\chi(a)}{\deg \chi},
\end{split}
\end{equation*}
\noindent for all $(a, \, \chi), (b, \, \chi^\prime) \in \Gamma$,
where  $G(a, \, b) = \{g \in G \mid agbg^{-1} = gbg^{-1}a\}$ and
$\overline{\chi}$ denotes the character conjugate to $\chi$.
For more details we refer the reader to \cite{BK}.

Note that the categorical dimensions \eqref{d} of objects of $\C$ are defined 
using the braiding on $\C$ and the above twist $\theta$.

\begin{remark}
It is known that the entries of the $S$-matrix lie in a cyclotomic field.
Also, the values of characters of a finite group are sums of roots of unity.
So we may assume that all scalars appearing herein are complex numbers;
in particular, complex conjugation and absolute values make sense. 
\end{remark}

The following lemma was proved in \cite[Lemma 3.1]{NN}.

\begin{lemma}
\label{centralize}
Two objects $(a, \, \chi), (b, \, \chi^\prime) \in \Gamma$ centralize each other
if and only if the following conditions hold:\\
(i) The conjugacy classes $K_a, K_b$ commute element-wise,\\
(ii) $\chi(gbg^{-1}) \chi^\prime(g^{-1}ag)=\deg \chi \; \deg \chi'$, 
for all $g \in G$.
\end{lemma}

We generalize this result in the next lemma.

\begin{lemma}
\label{centralize'}
For any two objects $(a, \, \chi), (b, \, \chi^\prime) \in \Gamma$,
$|S((a, \chi), (b, \chi'))| = d(a,\chi) \, d(b,\chi')$ 
if and only if the following conditions hold:\\
(i) The conjugacy classes $K_a, K_b$ commute element-wise,\\
(ii) $|\chi(gbg^{-1})| = \deg \chi$, $|\chi^\prime(g^{-1}ag)|=\deg \chi'$,
and $\chi(gbg^{-1}) \chi'(g^{-1}ag) = \chi(b) \chi'(a)$, for all $g \in G$.
\end{lemma}
\begin{proof}
The condition $|S((a, \chi), (b, \chi'))| = d(a,\chi) \, d(b,\chi')$
is equivalent to the condition 
\begin{equation}
\label{eqn}
\left| \sum_{g \in G(a, \, b)}
\chi(gbg^{-1}) \, \chi^\prime(g^{-1}ag) \right| = 
|G| \, \deg \chi \, \deg \chi^\prime,
\end{equation}
where $G(a, \, b) = \{g \in G \mid agbg^{-1} = gbg^{-1}a\}$.
It is clear that if the two conditions in the statement 
of the lemma hold, then \eqref{eqn} holds. 

Now suppose that \eqref{eqn} holds. 
We will show that this implies the two conditions in the statement of
the lemma. We have
\begin{equation*}
\begin{split}
|G| \, \deg \chi \, \deg \chi^\prime 
&= \left| \sum_{g \in G(a, \, b)} \chi(gbg^{-1}) \, \chi^\prime(g^{-1}ag)
\right|\\
&\leq \sum_{g \in G(a, \, b)} |\chi(gbg^{-1})| \, |\chi^\prime(g^{-1}ag)|\\
&\leq |G| \, \deg \chi \, \deg \chi^\prime.
\end{split}
\end{equation*}
So
$\sum_{g \in G(a, \, b)} |\chi(gbg^{-1})| \, |\chi^\prime(g^{-1}ag)|
= |G| \, \deg \chi \, \deg \chi^\prime$. Since 
$$
|G(a, \, b)| \leq |G|,\\
\,\,|\chi(gbg^{-1})| \leq \deg \chi, \mbox{ and } 
|\chi^\prime(g^{-1}ag)| \leq \deg \chi^\prime, 
$$
we must have
$G(a, \, b) = G$, 
$|\chi(gbg^{-1})| = \deg \chi$, and
$|\chi^\prime(g^{-1}ag)| = \deg \chi^\prime$. 
The equality $G(a, \, b) = G$
implies that the conjugacy classes $K_a, K_b$ commute element-wise. 
Since
$|\chi(gbg^{-1})| = \deg \chi$ and
$|\chi^\prime(g^{-1}ag)| = \deg \chi^\prime$, there exist roots of
unity $\alpha_g$ and $\beta_g$ such that 
$\chi(gbg^{-1}) = \alpha_g \, \deg \chi$ and
$\chi^\prime(g^{-1}ag) = \beta_g \, \deg \chi^\prime$, for all $g \in G$.
Substitute in \eqref{eqn} to obtain the equation 
\begin{equation}
\label{a}
\left|\sum_{g \in G} \alpha_g \beta_g \right| = |G|.
\end{equation}
Note that \eqref{a} holds if and only if 
$\alpha_g \beta_g = \alpha_e \beta_e$, for all $g \in G$. 
This is equivalent to
saying that $\chi(gbg^{-1}) \, \chi^\prime(g^{-1}ag) 
= \chi(b) \, \chi^\prime(a)$,
for all $g \in G$, and the lemma is proved.
\end{proof}

\begin{note}
\label{centralize''}
The following special cases of Lemma \ref{centralize'} will be used later.\\
(i) $|S((e, \chi), (b, \chi'))| = d(e, \chi) \, d(b,\chi')$ if and only
if $\chi|_{[G,b]} = \deg \chi$, where $$[G,b]=\langle gbg^{-1}b^{-1} \mid g \in G\rangle .$$\\
(ii) $|S((a, 1), (b, \chi'))| = d(a,1) \, d(b,\chi')$ if and only
if the conjugacy classes $K_a, K_b$ commute element-wise and
$\chi'|_{[G,a]} = \deg \chi'$.
\end{note}

\end{subsection}
%%%%%%%%%%%%%%%%%%%%%%%%%%%%%%%%%%%%%%%%%%%%%%%%%%%%%%%%%%%%%%%%%%%%%%%%%%%%%%%%%%%%%%%%%%%%%%%%%

%%%%%%%%%%%%%%%%%%%%%%%%%%%%%%%%%%%%%%%%%%%%%%%%%%%%%%%%%%%%%%%%%%%
\begin{subsection}{Canonical data determined by a fusion subcategory of $\mathbf{\Rep(D(G))}$}
\label{D --> triple}

Let $\D$ be a fusion subcategory of $\C$. There are two canonical
normal subgroups of $G$ determined by $\D$. First, let
\begin{equation}
K_\D := \{gag^{-1}\mid g\in G\,
        \text{ and } (a,\chi)\in \D \cap \Gamma \text{ for some }\chi \}
\end{equation}
be the {\em support} of $\D$. That is, $K_\D$ is a minimal subgroup of $G$
with the property that every vector bundle in $\D$ is supported on $K_\D$.

Second, let $H_\D$ be the normal subgroup of $G$ such that
$\D \cap \Rep(G) = \Rep(G/H_\D)$. Equivalently, 
\begin{equation}
H_\D:= \bigcap_{\chi: (e,\, \chi)\in \D}\, \mbox{Ker}(\chi), 
\end{equation}
where $\mbox{Ker}(\chi)=\{g\in G \mid \chi(g) = \deg \chi \}$ is
the kernel of the character $\chi$.

\begin{proposition}
\label{HD'=KD}
Let $\D$ be a fusion subcategory of $\C$.
Then  $H_{\D'} =K_\D$, $K_{\D'} =H_\D$, and the
subgroups $K_\D$ and $H_\D$ centralize each other.
\end{proposition}
\begin{proof}
%The first equality says that $\D'\cap \Rep(G) = \Rep(G/K_\D)$.
Consider a simple object $(e,\, \chi)$ of $\Rep(G)$.
By Lemma~\ref{centralize} it centralizes $\D$ if and only if
$\chi(a) =\deg\chi$ for all $a\in K_\D$. Thus,
$\D'\cap \Rep(G) = \Rep(G/K_\D)$, which is the first equality.

The second equality follows by replacing $\D$ with $\D'$ and using M\"uger's
double centralizer theorem (${\mathcal D}''={\mathcal D}$)
\cite{M1}. By Lemma~\ref{centralize}, $K_\D$ and $K_{\D'}$ 
centralize each other.
\end{proof}

%\begin{corollary}
%\label{HD centralizes KD}
%For any subcategory $\D \subseteq \C$ the canonical normal subgroups
%$K_\D$ and $H_\D$ centralize each other. 
%\end{corollary}
%%
%\begin{proof}
%By Lemma~\ref{centralize} $K_\D$ and $K_{\D'}$ centralize each other.
%\end{proof}

\begin{lemma}
\label{noname}
Let $\E$ be a fusion subcategory of $\Rep(D(G))$ and let
$(a,\chi) \in \E \cap \Gamma$. Suppose $\chi'$ is an irreducible
character of $C_G(a)$ such that $\frac{\chi'|_{K_{\E'}}}{\deg \chi'} 
= \frac{\chi|_{K_{\E'}}}{\deg \chi}$, where $K_{\E'}$ is the 
support of $\E'$. Then $(a, \chi') \in \E$.
\end{lemma}
\begin{proof}
Let us show that $(a, \chi')$ centralizes $\E'$. Pick any $(b, \chi'') \in 
\E' \cap \Gamma$. To see that $(a, \chi')$ and $(b, \chi'')$ centralize
each other we only need to check that Condition (ii) of Lemma \ref{centralize}
is satisfied. We have 
$$\frac{\chi'(gbg^{-1})}{\deg \chi'} \frac{\chi''(g^{-1}ag)}{\deg \chi''}
= \frac{\chi(gbg^{-1})}{\deg \chi} \frac{\chi''(g^{-1}ag)}{\deg \chi''}
=1,
$$
for all $g \in G$. The first equality above is due to the equality
$\frac{\chi'|_{K_{\E'}}}{\deg \chi'} = \frac{\chi|_{K_{\E'}}}{\deg \chi}$,
while the second equality is due to the fact that $(a, \chi')$ and $(b, \chi'')$
centralize each other. Therefore, $(a, \chi')$ centralizes $\E'$, i.e.,
$(a, \chi') \in \E''=\E$ and the lemma is proved.

\end{proof}

Let us define a pairing
\[
B_\D:  K_\D\times H_\D = K_\D \times K_{\D'} \rightarrow k^{\times},
\]
as follows. Let $(a,\chi) \in \D \cap \Gamma$, and
\begin{equation}
\label{Sarah's B}
   B_\D(g^{-1}ag,h) := \frac{\chi(ghg^{-1})}
                    {\deg \chi},
\end{equation}
for all $g \in G, h \in H_{\D}$.
Let $(b,\chi')\in \D' \cap \Gamma$, so that
$(b,\chi')$ centralizes $(a,\chi)$.
Then
\begin{equation}\label{alternative BD}
   \frac{\chi(gbg^{-1})}{\deg\chi} = \left(\frac{\chi'(g^{-1}ag)}
   {\deg\chi '}\right)^{-1}
\end{equation}
for all $g\in G$.
The equation (\ref{alternative BD}) shows that the pairing does not depend on the choice of $\chi$.
By its definition, $B_{\mathcal D}$ is $G$-invariant.
Also note that 
$$
  B_\D(kk',h) = B_\D(k,h) B_\D(k',h) \ \ \ \mbox{ and } \ \ \
   B_\D(k,hh')=B_\D(k,h) B_\D(k,h')
$$
for all $k,k'\in K_\D$ and $h,h'\in H_\D =K_{\D'}$; 
the latter equality is immediate from the definition of $B_{\mathcal D}$,
while the former follows from equation (\ref{alternative BD}) which gives an alternative
definition of $B_{\mathcal D}$.

Thus, to every fusion subcategory $\D\subseteq \C$ we associate
a canonical triple $(K_\D,\, H_\D,\, B_\D)$, where $K_\D$ and $H_\D$
are normal subgroups of $G$ centralizing each other and
$B_\D: K_\D \times H_\D \to k^\times$ is a $G$-invariant bicharacter.
\end{subsection}
%%%%%%%%%%%%%%%%%%%%%%%%%%%%%%%%%%%%%%%%%%%%%%%%%%%%%%%%%%%%%%%%%%%%%%%%%%%%%%%%%

%%%%%%%%%%%%%%%%%%%%%%%%%%%%%%%%%%%%%%%%%%%%%%%%%%%%%%%%%%%%%%%%%%%%%%%%%%%%%%%%%
\begin{subsection}{Construction of a fusion subcategory of $\Rep(D(G))$}
\label{triple --> D}
 
Suppose we have two normal subgroups $K$ and $H$ of $G$ that
centralize each other, and a $G$-invariant bicharacter $B: K\times H
\rightarrow k^{\times}$.
Define
\begin{equation}
\label{definition of S(K,H,B)}
\begin{split}
\S{(K,H,B)} := 
&\text{ full Abelian subcategory of } \C  \text{ generated by } \\
&\left\{(a, \, \chi) \in \Gamma \,\, \vline \,
\begin{tabular}{l}
$a \in K \cap R \text{ and } \chi \text{ is an irreducible character of } C_G(a)$ \\
$\text{ such that } \chi(h) = B(a, \, h) \,\deg \chi, \text{ for all } h \in H$
\end{tabular}
\right\}.
\end{split}
\end{equation}

\begin{remark}
\label{S as bundles} 
\begin{enumerate}
\item[(i)] As we mentioned in Remark~\ref{C = bundles}, $\Rep(D(G))$ is identified
with the category of $G$-equivariant bundles on $G$. The above subcategory
$\S{(K,H,B)}$ is identified  with the subcategory of bundles %(faithfully)
supported on $K$
whose $G$-equivariant structure restricts to $H$ as follows: The action of $h\in H$
on the fiber corresponding to  $a\in K$ is the scalar multiplication by $B(a,\, h)$. 
That $K$ is indeed the support of $\S{(K,H,B)}$ follows from the next lemma.\\
\item[(ii)] Note that $\S(G,\{e\},1) = \Rep(D(G))$, while $\S(\{e\},G,1)$
is the trivial fusion subcategory of $\Rep(D(G))$.
\end{enumerate}
\end{remark}

The following lemma was proved in \cite[Lemma 3.2]{NN}.

\begin{lemma}
\label{FR}
Let $E$ be a normal subgroup of a finite group $F$. 
Let $\Irr(F)$ denote the set of irreducible characters of $F$.
Let $\rho$ be an $F$-invariant character of $E$ of degree 1. Then
$$
\sum_{\chi \in \Irr(F) : \chi|_E = (\deg \chi) \, \rho} 
(\deg \chi)^2 = [F:E].
$$
\end{lemma}

Recall that by the dimension of a full  Abelian subcategory of $\C$ we 
mean the sum of  squares of dimensions of its simple objects. 

\begin{lemma}
\label{dim S(K,H,B)}
The dimension of the subcategory $\S(K,\, H,\, B)$ is $|K| [G:H]$.
\end{lemma}
\begin{proof}
We compute
\begin{equation*}
\begin{split}
\dim(\S(K,H,B))  
&= \sum_{(a,\chi) \in \S(K,H,B) \cap \Gamma} \dim((a,\chi))^2 \\
% &= \sum_{(a,\chi) \in \S(K,H,B) \cap \Gamma} |K_a|^2 \, (\deg \chi)^2\\
&= \sum_{a \in K \cap R} |K_a|^2 \sum_{\chi :
(a, \, \chi) \in \S(K,H,B) \cap \Gamma} (\deg \chi)^2\\
&= \sum_{a \in K \cap R} |K_a|^2 [C_G(a):H] = |K|[G:H],
% &= \frac{|G|}{|H|} \sum_{a \in K \cap R} |K_a|\\
%&= |K|[G:H].
\end{split}
\end{equation*}
where the third equality above is explained as follows.
Fix $a \in K \cap R$ and 
observe that $B(a,\cdot)$ is a $C_G(a)$-invariant
character of $H$ of degree $1$ and then apply Lemma \ref{FR}.
\end{proof}

Given a bicharacter $B: K\times H \to k^\times$ let us define
$B^{\op}: H\times K \to k^\times$ by $B^{\op}(h,\,k) = B(k,\,h)$
for all $k\in K,\, h\in H$.

\begin{lemma}
\label{centralizer of S(K,H,B)}
$\S(K,H,B)$ is a fusion subcategory of $\C$
and  $\S(K,H,B)' = \S(H,K,(B^{\op})^{-1})$.
\end{lemma}
\begin{proof}
First, we show that the full Abelian subcategories 
$\S(K,H,B),\S(H,K,(B^{\op})^{-1}) \subseteq \C$
centralize each other. Let $(a,\chi) \in \S(K,H,B)$
and $(b,\chi') \in \S(H,K,(B^{\op})^{-1})$. Since
$K$ and $H$ centralize each other, in order to show that
$(a,\chi)$ and $(b,\chi')$ centralize each other, it only remains
to check that condition (ii) of Lemma \ref{centralize} holds. We have
$$
\frac{\chi(gbg^{-1})}{\deg \chi} \frac{\chi^\prime(g^{-1}ag)}{\deg \chi'}
=B(a,gbg^{-1}) (B^{\op})^{-1}(b,g^{-1}ag) = B(a,gbg^{-1})B(g^{-1}ag,b)^{-1} = 1,
$$
for all $g \in G$. The first equality above is by definition of
$\S(K,H,B)$ and $\S(H,K,(B^{\op})^{-1})$, while the last equality
is due to $G$-invariance of $B$.
Therefore, condition (ii) of Lemma \ref{centralize}
holds and it follows that $\S(K,H,B)$ and $\S(H,K,(B^{\op})^{-1})$
centralize each other.  

By Lemma~\ref{dim S(K,H,B)}, $\dim(\S(K,H,B)) \cdot \dim{(\S(H,K,(B^{\op})^{-1})}
= |K| [G: H] \cdot |H|[G:K]  = |G|^2 = \dim(\C)$. Since $\S(K,H,B)$ and 
$\S(H,K,(B^{\op})^{-1})$ centralize each other it follows from 
Lemma \ref{abelian subcat. is fusion} that $\S(K,H,B)$ is a fusion 
subcategory and $\S(K,H,B)' = \S(H,K,(B^{\op})^{-1})$.
\end{proof}

\begin{theorem}
\label{bijection}
The assignments $\D \mapsto (K_\D,H_\D,B_\D)$  and
$(K,H,B) \mapsto \S(K,H,B)$ are inverses of each other.
Thus, there is a bijection between the set of fusion subcategories
of $\C$ and triples $(K,H,B)$, where $K,H\subseteq G$ are
normal subgroups of $G$ centralizing each other and $B: K\times H\to k^\times$
is a $G$-invariant bicharacter.
\end{theorem}

\begin{proof}
First, we show that $(K_{\S(K,H,B)},H_{\S(K,H,B)},B_{\S(K,H,B)})
= (K,H,B)$. As mentioned in Remark \ref{S as bundles}, it follows from Lemma \ref{FR}
that  $K_{\S(K,H,B)} = K$.
By  Lemma~\ref{centralizer of S(K,H,B)},
$H_{\S(K,H,B)}$ is the support of  $\S(K,H,B)'$ and so
$H_{\S(K,H,B)} = H$.
For any $a \in K \cap R, h \in H$, we have $B_{\S(K,H,B)}(a, h)
= \frac{\chi(h)}{\deg \chi}$, where $(a,\chi) \in \S(K,H,B)\cap \Gamma$.
Since $(a,\chi) \in \S(K,H,B)$, we have $\chi(h) = B(a,h) \deg \chi$.
Therefore, $B_{\S(K,H,B)}(a,h) = B(a,h)$ and the $G$-invariance of the
two bicharacters in question imply that $B_{\S(K,H,B)}=B$.

Second, we show that $\S(K_\D,H_\D,B_\D)=\D$. Pick any $(a,\chi) \in \S(K_\D,H_\D,B_\D)$.
Then for all $h\in H_\D$, 
$$
\frac{\chi(h)}{\deg \chi}=B_\D(a,h) = \frac{\chi'(h)}{\deg \chi'},
$$
where $(a,\chi') \in \D$. Here the first equality follows from
the definition of $\S(K_\D,H_\D,B_\D)$ \eqref{definition of S(K,H,B)}, 
while the second equality follows from the definition of $B_\D$ \eqref{Sarah's B}. 
Therefore,  $\frac{\chi|_{K_{\D'}}}{\deg \chi} = \frac{\chi'|_{K_{\D'}}}{\deg \chi'}$
and it follows from Lemma \ref{noname} that $(a,\chi) \in \D$.
So, $\S(K_\D,H_\D,B_\D) \subseteq \D$. Similarly, 
$\S(K_{\D'},H_{\D'},B_{\D'}) = \S(H_{\D}, K_\D, B_{\D'}) \subseteq \D'$. 
Using  Lemma \ref{dim S(K,H,B)} we have
\[
\dim(\C) = |G|^2 =   \dim(\S(K_\D,H_\D,B_\D)) \cdot 
\dim(\S(H_{\D}, K_\D, B_{\D'})) \leq \dim(\D)\dim(\D') =\dim(\C). 
\]
Hence, $\dim(\S(K_\D,H_\D,B_\D))= \dim(\D)$ and so $\S(K_\D,H_\D,B_\D)=\D$.
\end{proof}

\end{subsection}
%%%%%%%%%%%%%%%%%%%%%%%%%%%%%%%%%%%%%%%%%%%%%%%%%%%%%%%%%%%%%%%%%%%%%%%%%%%%%%%%%%%%%%%%%%%%%%%%%

\end{section}
%%%%%%%%%%%%%%%%%%%%%%%%%%%%%%%%%%%%%%%%%%%%%%%%%%%%%%%%%%%%%%%%%%%%%%%%%%%%%%%%%%%%%%%%%%%%%%%%%

%%%%%%%%%%%%%%%%%%%%%%%%%%%%%%%%%%%%%%%%%%%%%%%%%%%%%%%%%%%%%%%%%%%%%%%%%%%%%%%%%%%%%%%%%%%%%%%%%
\begin{section}
{Some invariants and the lattice of fusion subcategories of $\Rep(D(G))$}

%%%%%%%%%%%%%%%%%%%%%%%%%%%%%%%%%%%%%%%%%%%%%%%%%%%%%%%%%%%%%%%%%%%%%%%%%%%%%%%%%%%%%%%%%%%%%%%%%
\begin{subsection}{The lattice}\label{sec:lattices}
Let $G$ be a finite group. Let $K, H, K'$, and $H'$ be normal subgroups of $G$ 
such that $K$ and $H$ centralize each other and $K'$ and $H'$ centralize
each other. Let $B:K \times H : \to k^\times$  and
$B':K' \times H' : \to k^\times$ be  $G$-invariant bicharacters.

\begin{proposition}
\label{prop:subcat}
$\S(K,H,B) \subseteq \S(K',H',B')$ if and only if 
$K \subseteq K', \, H' \subseteq H$, and $B|_{K \times H'}
= B'|_{K \times H'}$.
\end{proposition}
\begin{proof}
The statement follows from the definitions and Remark \ref{S as bundles}.
\end{proof}

Define a homomorphism 
$$
\varphi_{B,B'}:K \cap K' \to \widehat{H \cap H'} \ \ \ 
  \mbox{ by } \ \ \ a \mapsto 
B(a, \cdot)|_{H \cap H'} (B')^{-1}(a, \cdot)|_{H \cap H'},
$$
where $\widehat{H\cap H'}$ denotes the group of group homomorphisms
from $H\cap H'$ to $k^{\times}$.
Also, define a $G$-invariant bicharacter
$$
\psi_{B,B'}: \ker \varphi_{B,B'} \times HH' \to k^\times  \ \ \ \mbox{ by } \ \ \ 
(a, hh') \mapsto B(a,h)B'(a,h').
$$

\begin{proposition}
\label{intersection and join}
\begin{enumerate}
\item [(i)] $\S(K,H,B) \cap \S(K',H',B') = \S(\ker \varphi_{B,B'}, HH', \psi_{B,B'})$.
\item [(ii)] $\S(K,H,B) \vee \S(K',H',B')= \S(KK', \ker \varphi_{B^{\op}, (B')^{\op}},
(\psi_{B^{\op}, (B')^{\op}})^{\op})$.
\end{enumerate}
\end{proposition}
\begin{proof}
(i) First, we show that $\S(K,H,B) \cap \S(K',H',B') \subseteq
\S(\ker \varphi_{B,B'}, HH', \psi_{B,B'})$. Let $(a, \chi)$
be any simple object of $\S(K,H,B) \cap \S(K',H',B')$.
By definition, $a \in K \cap K' \cap R$ and $\chi$ is an
irreducible character of $C_G(a)$ such that 
$$
\chi|_H = B(a, \cdot) \deg \chi \text{ and } \chi|_{H'} = 
B'(a, \cdot) \deg \chi.
$$
It follows that 
$B(a, \cdot)|_{H \cap H'} = B'(a, \cdot)|_{H \cap H'}$,
so that $a \in \ker \varphi_{B,B'}$. Also,  
$$
\chi(hh')
= B(a,h)B'(a,h') \deg \chi = \psi_{B,B'}(a, hh') \deg \chi,
$$ 
for all $h \in H, h' \in H'$. Therefore, 
$(a, \chi) \in \S(\ker \varphi_{B,B'}, HH', \psi_{B,B'})$,
so $\S(K,H,B) \cap \S(K',H',B') \subseteq
\S(\ker \varphi_{B,B'}, HH', \psi_{B,B'})$.

Next, we show that $\S(\ker \varphi_{B,B'}, HH', \psi_{B,B'})
\subseteq \S(K,H,B) \cap \S(K',H',B')$. Let $(a, \chi)$
be a simple object in $\S(\ker \varphi_{B,B'}, HH', \psi_{B,B'})$.
By definition, $a \in \ker \varphi_{B,B'} \cap R
\subseteq K \cap K' \cap R$ and $\chi$ is an
irreducible character of $C_G(a)$ such that 
$$
\chi|_{HH'} = \psi_{B,B'}(a, \cdot) \deg \chi.
$$
If $h \in H$, this implies 
$$
\chi(h) = \psi_{B,B'}(a, h) \deg \chi = B(a,h) B'(a,e) \deg \chi = B(a,h) \deg \chi.
$$
Similarly, if $h' \in H'$, then $\chi(h') = B'(a,h') \deg \chi$.
By definition,  $(a, \chi) \in \S(K,H,B) \cap \S(K',H',B')$,
so $\S(\ker \varphi_{B,B'}, HH', \psi_{B,B'})
\subseteq \S(K,H,B) \cap \S(K',H',B')$ and the proposition is proved.

(ii) We have
\begin{equation*}
\begin{split}
\S(K,H,B) \vee \S(K',H',B')
&= \left(\S(K,H,B)' \cap \S(K',H',B')' \right)'\\
&= \left(\S(H,K,(B^{\op})^{-1}) \cap \S(H',K',((B')^{\op})^{-1}) \right)'\\
&= \S\left(\ker \varphi_{(B^{\op})^{-1}, ((B')^{\op})^{-1}}, KK', 
\psi_{(B^{\op})^{-1}, ((B')^{\op})^{-1}} \right)'\\
&= \S\left( KK', \ker \varphi_{(B^{\op})^{-1}, ((B')^{\op})^{-1}},
((\psi_{(B^{\op})^{-1}, ((B')^{\op})^{-1}})^{\op})^{-1} \right)\\       
&= \S\left( KK', \ker \varphi_{B^{\op}, (B')^{\op})},
(\psi_{B^{\op}, (B')^{\op}})^{\op} \right).
\end{split}
\end{equation*}
The first equality above is due to \cite[Lemma 2.8 and Theorem 3.2(i)]{M1},
the last equality follows from a direct calculation, and the other
equalities follow from either Lemma \ref{centralizer of S(K,H,B)} or
part (i).
\end{proof}

A braided tensor category $\C$ is said to be {\em symmetric} if the square 
of the braiding is the identity \cite{BK}.
A fusion subcategory $\D$ of a premodular category $\C$ with
positive categorical dimensions is said to be {\em isotropic} 
if the twist of $\C$ restricts to the
identity on $\D$ \cite{DGNO}. By Deligne's theorem \cite{D}
an isotropic subcategory $\D\subset \C$ is
{\em Tannakian}, i.e., it is equivalent to 
the representation category of a finite group as a symmetric
fusion category. An isotropic subcategory $\D\subseteq \C$
is said to be {\em Lagrangian} if $\D'=\D$, or
equivalently $(\dim(\D))^2 = \dim(\C)$. 
Recall that a bicharacter $f:L \times L \to k^\times$ on an abelian 
group $L$ is {\em alternating} if $f(x,x)=1$, for all $x \in L$.

\begin{proposition}\label{prop:sil}
The fusion subcategory $\S(K,H,B) \subseteq \Rep(D(G))$ is
\begin{enumerate}
\item[(i)] symmetric if and only if $K \subseteq H$ and
$B(k_1, k_2) B(k_2, k_1) =1$, for all $k_1, k_2 \in K$,
\item[(ii)] isotropic if and only if $K \subseteq H$ and
$B|_{K \times K}$ is alternating,
\item[(iii)] Lagrangian if and only if $K=H$ and $B$ is alternating.
\end{enumerate}
\end{proposition}
\begin{proof}
(i) ${\mathcal S}(K,H,B)$ is symmetric if and only if
${\mathcal S}(K,H,B) \subseteq {\mathcal S}(K,H,B)'$.
By Lemma \ref{centralizer of S(K,H,B)} and Proposition \ref{prop:subcat}, 
this is true if and only if $K\subseteq H$ and $B|_{K\times K}
=(B^{\op})^{-1}|_{K\times K}$, as stated.

(ii) Suppose $K\subseteq H$ and $B|_{K\times K}$ is alternating.
In particular, this implies $B(a,a)=1$ for all $a\in K\cap R$,
so that
$$
  \theta(a,\chi) =\frac{\chi(a)}{\deg\chi} = B(a,a) =1
$$
for all $(a,\chi)\in {\mathcal S}(K,H,B)\cap\Gamma$.
That is, ${\mathcal S}(K,H,B)$ is isotropic.

Conversely, assume ${\mathcal S}(K,H,B)$ is isotropic, so in
particular $\S(K,H,B)$ is symmetric. Therefore, by (ii),
$K \subseteq H$ and $B(k_1,k_2) B(k_2, k_1) = 1$, for all 
$k_1,k_2 \in K$. It remains to show that $B(k,k) = 1$, for all
$k \in K$; to this end, pick any $(a, \chi) \in \S(K,H,B) \cap \Gamma$
and observe that
$$ B(a,a) = \frac{\chi(a)}{\deg \chi} = \theta(a,\chi) = 1, $$
where the last equality is due to the definition of isotropic. Normality
of $K$ and $G$-invariance of $B$ together imply that 
$B(k,k) = 1$, for all $k \in K$, as desired.

(iii) This was proved in \cite{NN}.
\end{proof}

Recall that the notions of M\"uger's center $\Z_2$, nondegenerate and prime
braided fusion categories were defined in Section~\ref{prelim: centralizers}.
Also, recall that a symmetric bicharacter $f:L \times L \to k^\times$ on a 
group $L$ is {\em nondegenerate} if $\{x \in L \mid f(x,y)=1 \; \mbox{ for all }y \in L\} = \{e\}$. 

\begin{proposition}
\label{prop:Mueger center of S(K,H,B)}
$\Z_2(\S(K,H,B)) = \S(\ker \varphi_{B,(B^{\op})^{-1}}, HK, \psi_{B,(B^{\op})^{-1}})$.
\end{proposition}
\begin{proof}
Using Lemma \ref{centralizer of S(K,H,B)} and Proposition \ref{intersection and join} (i), 
we have
\begin{equation*}
\begin{split}
\Z_2(\S(K,H,B))
&= \S(K,H,B) \cap \S(K,H,B)' \\
&= \S(K,H,B) \cap \S(H,K,(B^{\op})^{-1}) \\
&= \S(\ker \varphi_{B,(B^{\op})^{-1}}, HK, \psi_{B,(B^{\op})^{-1}}).
\end{split}
\end{equation*}
\end{proof}

\begin{proposition}
\label{S(K,H,B) is nondegenerate iff}
\begin{enumerate}
\item[(i)] The fusion subcategory $\S(K,H,B) \subseteq \Rep(D(G))$ 
is nondegenerate if and only if $HK=G$ and the symmetric bicharacter 
$BB^{\op}|_{(K \cap H) \times (K \cap H)}$
is nondegenerate.
\item[(ii)] $\Rep(D(G))$ is prime if and only if there is no triple
$(K,H,B)$, where $K$ and $H$ are normal subgroups of $G$ that 
centralize each other, $(G,\{e\}) \neq (K,H) \neq (\{e\},G)$, $HK = G$, and $B$ is a $G$-invariant bicharacter
on $K\times H$ such that the symmetric bicharacter
$BB^{\op}|_{(K \cap H) \times (K \cap H)}$ is nondegenerate.
\end{enumerate}
\end{proposition}
\begin{proof}
(i) Note that 
${\mathcal S}(K,H,B)$ is nondegenerate
if and only if $\Z_2(\S(K,H,B))=
{\mathcal S}(\{e\},G,1)$, by \cite[Corollary 2.16]{M1}. 
By Proposition \ref{prop:Mueger center 
of S(K,H,B)}, this is the case if and only if
$\ker\varphi_{B,(B^{\op})^{-1}}=\{e\}$, $HK=G$, and $\psi_{B,(B^{\op})^{-1}}=1$.
The condition $\ker\varphi_{B,(B^{\op})^{-1}}=\{e\}$ is equivalent to the
condition that
$B(a,\cdot)|_{H\cap K}B^{\op}(a,\cdot)|_{H\cap K}$ is nontrivial on
$H\cap K$ whenever $a\in H\cap K-\{e\}$, that is,
$BB^{\op}|_{(K\cap H)\times (K\cap H)}$ is nondegenerate.
The condition $\psi_{B,(B^{\op})^{-1}} =1$ is automatically true when
$\ker\varphi_{B,(B^{\op})^{-1}}=\{e\}$.

(ii) This follows immediately from Theorem \ref{bijection}, Remark~\ref{S as bundles} (ii), and
part (i).
\end{proof}

\begin{remark}
\label{p^n}
Let $G:= \mathbb{Z}/p^n\mathbb{Z}$, where $p$ is a  prime. It was shown
in \cite{M1} that for $p=2$ the category $\Rep(D(G))$ is prime and that for odd $p$  
nontrivial proper modular subcategories of $\Rep(D(G))$ are
in bijection with isomorphisms $G \xrightarrow{\sim} \widehat{G}$, where $\widehat{G}$
denotes the set of group homomorphisms from $G$ to $k^{\times}$.
Let us recover this result from  Proposition~\ref{S(K,H,B) is nondegenerate iff}.
Suppose that $\S(K,H,B)$ where $(G,\{e\}) \neq (K,H) \neq (\{e\},G)$ is a nondegenerate subcategory 
of  $\Rep(D(G))$. The condition $HK=G$ implies that at least one of the subgroups $H,\, K$ 
coincides with $G$.  Let us assume $H=G$ and suppose that $|K| = m < p^n$. 
Let $x,\,y\in G$ be generators of $H$ and $K$ respectively. Then $B(x,\,y)$ is 
an $m$th root of $1$, so $B(y,\, y)$ is a root of $1$ of order less than $m$
and $BB^{op}|_{K\times K}= B^2|_{K\times K}$ is degenerate. 
Thus we must have $H=K=G$ and $B$ must be a  (necessarily symmetric)
bicharacter on $G$ such that $B^2$ is nondegenerate.  This is impossible
for $p=2$.  When $p$ is odd $B^2$ is nondegenerate if and only if $B$ is nondegenerate,
i.e., when $B$ comes from an isomorphism $G \xrightarrow{\sim} \widehat{G}$.

%Let $G:= \mathbb{Z}/p^n\mathbb{Z}$, where $p$ is any odd prime. It was shown
%in \cite{M1} that non-trivial proper modular subcategories of $\Rep(D(G))$ are
%in bijection with isomorphisms $G \xrightarrow{\sim} \widehat{G}$. In what 
%follows we recover this result using Proposition~\ref{S(K,H,B) is nondegenerate iff}.
%To this end, we first note that the Schur multiplier
%$H^2(G,k^\times)$ of $G$ is trivial, so every bicharacter on $G$ is a coboundary and
%hence is symmetric. Also, note that since
%$|G|$ is odd, a bicharacter $B$ on $G$ is nondegenerate if and only if $B^2$ is
%nondegenerate. In view of the above remarks and thanks to the unique decomposition
%of abelian groups it follows from Theorem~\ref{bijection} and 
%Proposition~\ref{S(K,H,B) is nondegenerate iff}
%that non-trivial proper modular subcategories of $\Rep(D(G))$ are
%in bijection with nondegenerate bicharacters on $G$, which in turn are in
%bijection with isomorphisms $G \xrightarrow{\sim} \widehat{G}$,
%squaring with \cite{M1}.
\end{remark}

As was observed by M\"uger in \cite{M1}, prime nondegenerate categories serve as building blocks
for nondegenerate categories.
Precisely, it was shown in \cite[Theorem 4.5]{M1} that every nondegenerate category is
equivalent to a finite direct product of prime ones (note, however, that such a 
decomposition is not unique in general). Below we collect some remarks concerning
the primality of $\Rep(D(G))$.

\begin{remark}
\begin{enumerate}
\item[(i)] If $G = M \times N$ is a direct product of non-trivial groups $M$ and $N$,
then $\Rep(D(G))$ is not prime. Indeed, by Proposition~\ref{S(K,H,B) is nondegenerate iff},
the (nontrivial proper) subcategory $\S(M \times \{e\}, \{e\} \times N, 1) \subseteq 
\Rep(D(G))$ is nondegenerate. This fact was also observed by M\"uger in \cite{M1}.
\item[(ii)] If $G$ is simple and nonabelian, then it immediately follows from
Proposition~\ref{S(K,H,B) is nondegenerate iff} that $\Rep(D(G))$ is prime. Note that the
aforementioned statement is not true for abelian simple groups \cite{M1} (see also Remark~\ref{p^n}). 
\item[(iii)] As mentioned above, simplicity of a nonabelian group $G$ is sufficient
for $\Rep(D(G))$ to be prime, however, it is not necessary.
Indeed, take $G= S_n$, $n \geq 3$, and let $K,H$ be a pair of centralizing
normal subgroups of $G$ such that $HK = G$. Note that $K \cap H$, being a
central subgroup of $G$, must be trivial, as the center of $G$ is trivial.
This implies that $G$ is a direct product of $H$ and $K$. However, as is well
known, $S_n$ does not admit a non-trivial direct product decomposition,
therefore, $\Rep(D(S_n))$ is prime by Proposition~\ref{S(K,H,B) is nondegenerate iff}.
\end{enumerate}
\end{remark}

\end{subsection}
%%%%%%%%%%%%%%%%%%%%%%%%%%%%%%%%%%%%%%%%%%%%%%%%%%%%%%%%%%%%%%%%%%%%%%%%%%%%%%%%%%%%%%%%%%%%%%%%%

%%%%%%%%%%%%%%%%%%%%%%%%%%%%%%%%%%%%%%%%%%%%%%%%%%%%%%%%%%%%%%%%%%%%%%%%%%%%%%%%%%%%%%%%%%%%%%%%%
\begin{subsection}{The Gauss sum and central charge}\label{gscc}

Let $\D$ be a premodular category with twist $\theta$. 
Recall that the {\em Gauss sum} of $\D$ is defined by
$$ \tau(\D) = \sum_{X \in \Irr(\D)} \theta(X) d(X)^2, $$
where $\Irr(\D)$ is the set of isomorphism classes of simple objects in $\D$.
The {\em central charge} is defined by
$$ \zeta(\D) = \frac{\tau(\D)}{\sqrt{\dim(\D)}}. $$
For basic properties of Gauss sum and central charge 
we refer the reader to \cite[Sect.\ 3.1]{BK}. In particular,
the central charge of a non-degenerate braided fusion category
is known to be a root of unity (this statement is known as Vafa's 
theorem, see \cite[Thm. 3.1.19]{BK}).

\begin{proposition}\label{prop:gauss sum}
The Gauss sum of the fusion subcategory $\S(K,H,B) \subseteq \Rep(D(G))$
is 
$$
\tau(\S(K,H,B)) = \frac{|G|}{|H|} 
\sum_{a \in K \cap H \cap R} |K_a| B(a, a).
$$ 
When $\S(K,H,B)$ is nondegenerate its Gauss sum is
$$
\tau(\S(K,H,B)) = \frac{|K|}{|K \cap H|} 
\sum_{a \in K \cap H} B(a, a).
$$
and its central charge is
$$
\zeta(\S(K,H,B)) = \frac{1}{\sqrt{|K \cap H|}} 
\sum_{a \in K \cap H} B(a, a).
$$  
\end{proposition}
\begin{proof}
We have
\begin{equation*}
\begin{split}
\tau(\S(K,H,B)) 
&= \sum_{\text{simple } (a, \chi) \in \S(K,H,B)} 
\theta(a, \chi) \, d(a, \chi)^2 \\
&= \sum_{\text{simple } (a, \chi) \in \S(K,H,B)} 
\frac{\chi(a)}{\deg \chi} |K_a|^2 (\deg \chi)^2. 
\end{split}
\end{equation*}
By definition of $\S(K,H,B)$ the above expression is equal to
$$
\sum_{a \in K \cap R} \left( |K_a|^2 
\sum_{\chi \in \Irr(C_G(a)): \chi|_H = B(a, \cdot) \deg \chi}  
(\deg \chi) \chi(a) \right). 
$$
Observing that $B(a, \cdot)$ is
$C_G(a)$-invariant and applying Frobenius reciprocity and
Clifford theory, the above expression is equal to
$$
\sum_{a \in K \cap R} \left( |K_a|^2 
\sum_{\chi \text{ is an irr. constituent of } \Ind_H^{C_G(a)}B(a, \cdot)}  
(\deg \chi) \chi(a) \right).
$$
Applying Frobenius reciprocity again, the above expression is
equal to
$$
\sum_{a \in K \cap R} \left( |K_a|^2 \,
(\Ind_H^{C_G(a)}B(a, \cdot))(a) \right).
$$
Finally, using the formula
for the induced character and the fact that the induced character
is zero on $C_G(a) \backslash H$ the above expression is equal to
\begin{equation*}
\begin{split}
& \sum_{a \in K \cap H \cap R} \left( \frac{|K_a|^2}{|H|} 
\sum_{x \in C_G(a)} B(a, x^{-1}ax) \right) \\
= &\sum_{a \in K \cap H \cap R} \frac{|K_a|^2|C_G(a)|}{|H|}B(a,a)\\
= &\frac{|G|}{|H|} \sum_{a \in K \cap H \cap R} |K_a| B(a,a).
\end{split}
\end{equation*}
When $\S(K,H,B)$ is nondegenerate, by Proposition \ref{S(K,H,B) is nondegenerate iff},
$HK = G$, so $K \cap H$ is contained in the center of $G$. Therefore, the
Gauss sum is
\begin{equation*}
\begin{split}
\tau(\S(K,H,B))
&= \frac{|G|}{|H|} \sum_{a \in K \cap H \cap R} |K_a| B(a,a) \\
&= \frac{|K|}{|K \cap H|} \sum_{a \in K \cap H} B(a,a)
\end{split}
\end{equation*}
and the central charge is
\begin{equation*}
\begin{split}
\zeta(\S(K,H,B))
&= \frac{\tau(\S(K,H,B))}{\sqrt{\dim(\S(K,H,B))}}\\
&= \frac{\frac{|K|}{|K \cap H|} \sum_{a \in K \cap H} B(a,a)}
{\sqrt{|K| \cdot [G:H]}} \\
&= \frac{1}{\sqrt{|K \cap H|}} \sum_{a \in K \cap H} B(a,a). 
\end{split}
\end{equation*}
\end{proof}

\begin{remark}
Note that the sum $\sum_{a \in K \cap H} B(a,a)$ is the classical Gauss sum
for the quadratic form $a\mapsto  B(a,a)$ on the Abelian group $K \cap H$.
\end{remark}

\end{subsection}
%%%%%%%%%%%%%%%%%%%%%%%%%%%%%%%%%%%%%%%%%%%%%%%%%%%%%%%%%%%%%%%%%%%%%%%%%%%%%%%%%%%%%%%%%%%%%%%%%

%%%%%%%%%%%%%%%%%%%%%%%%%%%%%%%%%%%%%%%%%%%%%%%%%%%%%%%%%%%%%%%%%%%%%%%%%%%%%%%%%%%%%%%%%%%%%%%%%
\begin{subsection}{The central series}

Following Brugui\`{e}res \cite{Br}, we say that a functor $F:\C_1 \to \C_2$ between 
semisimple categories $\C_1$ and $\C_2$ is {\em dominant} if
every simple object in $\C_2$ is a direct summand of the image (under $F$) of
some object in $\C_1$. 

\begin{lemma}
\label{gens. of S(K,H,1)}
The fusion subcategory $\S(K,H,1) \subseteq \Rep(D(G))$ is
generated by the set
$gens(\S(K,H,1)) := \{ (a, 1) \in \Gamma \mid a \in K \cap R\} 
\cup \{ (e,\chi) \in \Gamma \mid \chi|_H = \deg \chi\}$.
\end{lemma}
\begin{proof}
The lemma follows immediately from the observation that the
restriction functors $\Rep(G/H) \to \Rep(C_G(a)/H)$, 
$a \in K \cap R$, are dominant.
\end{proof}

\begin{proposition}
\label{adjoint of S(K,H,1)}
The adjoint subcategory $\S(K,H,1)_{ad}$ of the fusion subcategory
$\S(K,H,1)$ of $\Rep(D(G))$ is 
$\S([G,K], C_G(K) \cap \pi^{-1}(Z(G/H)), 1)$, where $\pi$ is the
canonical surjection from $G$ to $G/H$ and $[G,K] = \langle gkg^{-1}k^{-1} \mid
g \in G, k \in K\rangle $.
\end{proposition}
\begin{proof}
By M\"uger's double centralizer theorem and Lemma 
\ref{centralizer of S(K,H,B)} it suffices to show that
$(\S(K,H,1)_{ad})' = \S(C_G(K) \cap \pi^{-1}(Z(G/H)), [G,K], 1)$.
By \cite[Proposition 6.7]{GN}, the simple objects (up to isomorphism)
of $(\S(K,H,1)_{ad})'$ are given by the set 
\begin{equation*}
\begin{split}
& \Irr((\S(K,H,1)_{ad})')\\
& = \{ (b, \chi') \in \Gamma \mid |S((a,\chi),(b,\chi'))| =
d(a,\chi) \, d(b,\chi'), \text{ for all simple } (a,\chi) \in \S(K,H,1) \}\\
& = \{ (b, \chi') \in \Gamma \mid |S((a,\chi),(b,\chi'))| =
d(a,\chi) \, d(b,\chi'), \text{ for all } (a,\chi) \in gens(\S(K,H,1)) \}\\
& = \{ (b, \chi') \in \Gamma \mid \chi|_{[G,b]}=\deg \chi,
\text{ for all } \chi \in \Irr(\Rep(G/H)), [K_a,K_b]=\{e\},\\
& \;\;\;\; \text{ and } \chi'|_{[G,a]} = \deg \chi',
\text{ for all }  a \in K \cap R\}\\
& = \{ (b, \chi') \in \Gamma \mid b \in C_G(K),
[G,b] \in H \text{ and } \chi'|_{[G,K]} = \deg \chi' \}\\
& = \{ (b, \chi') \in \Gamma \mid b \in C_G(K) \cap \pi^{-1}(Z(G/H)),
\text{ and } \chi'|_{[G,K]} = \deg \chi' \}\\
& = \Irr(\S(C_G(K) \cap \pi^{-1}(Z(G/H)), [G,K], 1)).
\end{split}
\end{equation*}
In the second equality above we used Lemma \ref{gens. of S(K,H,1)}
and in the third equality we used Note \ref{centralize''}.
\end{proof}

\begin{corollary}
$(i)$ The $n$th term of the upper central series of $\Rep(D(G))$ is given by
$$ \Rep(D(G))^{(n)} = \S(C_n(G), C_G(C_{n-1}(G)) \cap C^n(G), 1), \qquad n \geq 1. $$
$(ii)$ The $n$th term of the lower central series of $\Rep(D(G))$ is given by
$$ \Rep(D(G))_{(n)} = \S(C_G(C_{n-1}(G)) \cap C^n(G), C_n(G), 1), \qquad n \geq 1. $$
Here $C_n(G)$ and $C^n(G)$ are the $n$th terms of the lower and upper
central series of $G$, respectively.
\end{corollary}
\begin{proof}
Part $(i)$ follows immediately from Proposition \ref{adjoint of S(K,H,1)}, while
part $(ii)$ follows from \cite[Theorem 6.8]{GN} and Lemma \ref{centralizer of S(K,H,B)}.
\end{proof}

\end{subsection}
%%%%%%%%%%%%%%%%%%%%%%%%%%%%%%%%%%%%%%%%%%%%%%%%%%%%%%%%%%%%%%%%%%%%%%%%%%%%%%%%%%%%%%%%%%%%%%%%%

\end{section}
%%%%%%%%%%%%%%%%%%%%%%%%%%%%%%%%%%%%%%%%%%%%%%%%%%%%%%%%%%%%%%%%%%%%%%%%%%%%%%%%%%%%%%%%%%%%%%%%%

%%%%%%%%%%%%%%%%%%%%%%%%%%%%%%%%%%%%%%%%%%%%%%%%%%%%%%%%%%%%%%%%%%%%%%%%%%%%%%%%%%%%%%%%%%%%%%%%%
\begin{section}{Fusion subcategories of
the twisted quantum double of a finite group}
\label{Sect 5}

%%%%%%%%%%%%%%%%%%%%%%%%%%%%%%%%%%%%%%%%%%%%%%%%%%%%%%%%%%%%%%%%%%%%%%%%%%
\begin{subsection}
{The twisted quantum double of a finite group}

Recall that a normalized 3-cocycle
$\omega : G\times G\times G \rightarrow k^{\times}$ 
is a map  satisfying:
\begin{equation}
\label{3-cocycle condition}
\omega(g_2, \, g_3, \, g_4)\omega(g_1, \, g_2g_3, \, g_4)\omega(g_1, \, g_2, \, g_3)
= \omega(g_1g_2, \, g_3, \, g_4)\omega(g_1, \, g_2, \, g_3g_4),
\end{equation}
\begin{equation}
\omega(g, \, e, \, l) = 1,
\end{equation}
for all $g, l, g_1, g_2, g_3, g_4 \in G$.
It follows that $\omega(e,g,l)=1=\omega(g,l,e)$ for all $g,l\in G$ as well.
We may assume the values of $\omega$ are roots of unity.
Define 
\begin{eqnarray}
\label{beta-defn}
\beta_a(x,y)&=& \frac{\omega(a,x,y)\omega(x,y,y^{-1}x^{-1}
axy)}{\omega(x,x^{-1}ax,y)},\\
%\nonumber
\eta_a(x,y)&=& \frac{\omega(x,y,a) \omega(xyay^{-1}x^{-1}, x, y)}
{\omega(x, yay^{-1}, y)},\\
\label{gamma-defn}
\gamma_a(x,y) &=& \frac{\omega(x,y,a)\omega(a,a^{-1}xa,a^{-1}ya)}
{\omega(x,a,a^{-1}ya)},\\
\label{nu-defn}
\nu_a(x,y)&=& \frac{\omega(axa^{-1},aya^{-1},a) \omega(a,x,y)}
{\omega(axa^{-1},a,y)},
\end{eqnarray}
for all $a,x,y\in G$. 
Since $\omega$ is a 3-cocycle, we have
\begin{equation}
\label{beta-reln}
\beta_a(x,y)\beta_a(xy,z) = \beta_a(x,yz)\beta_{x^{-1}ax}(y,z)
\end{equation}
for all $a, x,y,z\in G$. Therefore, for any $a\in G$,
$\beta_a|_{C_G(a)}$ is a 2-cocycle.
Note that 
\begin{equation}\label{four-relns}
   \beta_a|_{C_G(a)}=\eta_a|_{C_G(a)}=\gamma_a|_{C_G(a)}
    =\nu_a|_{C_G(a)}.
\end{equation}

Direct calculations, using (\ref{3-cocycle condition}), show that
the following identities hold:

\begin{eqnarray}
\label{gamma beta relation 1}
\frac{\gamma_{ab}(x,y)} {\gamma_{b}(a^{-1}xa, a^{-1}ya)\gamma_a(x,y)}
&=& \frac{\beta_x(a,b)\beta_{y}(a,b)}{\beta_{xy}(a,b)},\\
\label{nu-eta-reln}
\frac{\nu_{ab}(x,y)}{\nu_a(bxb^{-1},byb^{-1}) \nu_b(x,y)} 
&=& \frac{\eta_x(a,b) \eta_y(a,b)}{\eta_{xy}(a,b)},
\end{eqnarray}
for all $a,b,x,y\in G$. 

Let $\C=\Rep(D^{\omega}(G))$. It is well known that $\C$ is a modular 
category  \cite{AC2}.
The set of isomorphism classes of simple objects in $\C$ may be
identified with the set
$$
  \Gamma := \{(a,\chi)\mid a\in R \mbox{ and } \chi \mbox{ is an
    irreducible }\beta_a\mbox{-character of }C_G(a)\},
$$
where a $\beta_a$-character of $C_G(a)$ is the trace function of a
$k^{\beta_a}[C_G(a)]$-module.

The following lemma is proved in \cite[Lemma 4.2]{NN}.

\begin{lemma}\label{lem-twisted}
Two objects $(a,\chi), (b,\chi')\in \Gamma$ centralize each other
if and only if the following conditions hold:
\begin{itemize}
\item[(i)] The conjugacy classes $K_a,K_b$ commute element-wise.
\item[(ii)] For all $x,y\in G$,

$\hspace{-2cm}\left(\frac{\beta_a(x,y^{-1}by)\beta_a(xy^{-1}by, x^{-1})
\beta_b(y,x^{-1}ax)\beta_b(yx^{-1}ax,y^{-1})}
{\beta_a(x,x^{-1})\beta_b(y,y^{-1})}\right)
\chi(xy^{-1}byx^{-1})\chi'(yx^{-1}axy^{-1}) = \deg \chi \deg\chi'$.
\end{itemize}
\end{lemma}

The lemma has the following useful consequence:
Taking the magnitude on both sides of the equation in
Lemma \ref{lem-twisted}(ii), we obtain
$$
  | \chi(xy^{-1}byx^{-1})\chi'(yx^{-1}axy^{-1})|
    =\deg\chi\deg\chi',
$$
since the values of $\beta_a$ and $\beta_b$ are roots of unity.
Now, for any $h\in C_G(a)$, $\chi(h)$ is a sum of $\deg \chi$ roots
of unity, and so by the triangle inequality, $|\chi(h)|\leq \deg\chi$,
with equality if and only if all roots are the same.
Similarly for $\chi'$.
Thus if $(a,\chi), (b,\chi')$ centralize each other, then
$\frac{\chi(xy^{-1}byx^{-1})}{\deg\chi}$ and $\frac{\chi'(yx^{-1}axy^{-1})}
{\deg\chi'}$ are roots of unity for all $x,y\in G$.

\begin{comment}
Another consequence of the lemma is that for any normal subgroup $K$
of $G$, the subcategories $\C_K$ and $\Rep(G/K)$ of $\C$ centralize each
other: 
Let $(a,\chi)\in \C_K$ and $(1,\chi')\in \Rep(G/K)$.
Then all conjugates of $a$ are in the kernel of $\chi'$.
Since $\omega$ is normalized, the equation in Lemma \ref{lem-twisted}(ii)
now holds.
Further, $\C_K' = \Rep(G/K)$ and $\Rep(G/K)'=\C_K$ by 
Lemma \ref{abelian subcat. is fusion} since
$$
  \dim(\C_K)\dim(\Rep(G/K)) = |G||K|\cdot [G:K] =
   |G|^2 = \dim(\C).
$$
In particular, letting $K=\{e\}$, we have $\Rep(G)' = \Rep(G)$.
\end{comment}

\end{subsection}
%%%%%%%%%%%%%%%%%%%%%%%%%%%%%%%%%%%%%%%%%%%%%%%%%%%%%%%%%%%%%%%%%

%%%%%%%%%%%%%%%%%%%%%%%%%%%%%%%%%%%%%%%%%%%%%%%%%%%%%%%%%%%%%%%%%
\subsection{Canonical data determined by a fusion subcategory
of $\Rep(D^{\omega}(G))$}

Let $\D$ be a fusion subcategory of $\C$.
Define two normal subgroups of $G$ as follows:
\begin{eqnarray*}
  K_{\D} & = & \{gag^{-1}\mid g\in G, \ (a,\chi)\in \D\cap\Gamma
                 \mbox{ for some } \chi\}\\
  H_{\D} & = & \bigcap_{\chi : (e,\chi)\in\D} \Ker (\chi)
\end{eqnarray*}
where $\Ker(\chi) = \{g\in G\mid \chi(g)=\chi(e)\}$.
Note that the definition of $H_{\D}$ is equivalent to
$\D\cap \Rep(G) =\Rep(G/H_{\D})$ as before.

\begin{proposition}\label{prop-HDKD}
Let $\D$ be a fusion subcategory of $\C$.
Then $H_{\D'}=K_{\D}$, $K_{\D'}=H_{\D}$, and the groups $K_{\D}$ and
$H_{\D}$ centralize each other.
\end{proposition}
\begin{proof}
By Lemma~\ref{lem-twisted} a simple object $(e,\, \chi)$ in $\Rep(G)$ centralizes
$\D$ if and only if $\chi(a) =\deg\chi$ for all $a\in K_\D$. 
The rest is similar to the proof of  Proposition~\ref{HD'=KD}.
\end{proof}

\begin{lemma}\label{lem-E}
Let $\E$ be a fusion subcategory of $\C$ and let $(a,\chi)\in \E\cap\Gamma$.
Suppose $\chi'$ is an irreducible $\beta_a$-character of $C_G(a)$
such that $\frac{\chi'|_{K_{\E'}}}{\deg\chi'} = \frac{\chi|_{K_{\E'}}}
{\deg\chi}$. Then $(a,\chi')\in \E$.
\end{lemma}

\begin{proof}
We will show that $(a,\chi')$ centralizes $\E'$.
Let $(b,\chi'')\in\E'\cap\Gamma$. 
Note that since $(a,\chi)\in\E$, the conjugacy classes $K_a$ and $K_b$
commute.
It remains to show that the equation in Lemma \ref{lem-twisted}(ii)
holds for the pair $(a,\chi'),(b,\chi'')$.
By hypothesis,  $\frac{\chi'|_{K_{\E'}}}{\deg\chi'} = \frac{\chi|_{K_{\E'}}}
{\deg\chi}$, so we may replace $\chi'$ by $\chi$ in the equation.
However we already know that $(a,\chi)$ centralizes $\E'$ since 
$(a,\chi)\in \E$, so by Lemma \ref{lem-twisted}(ii), the desired equation
holds.
\end{proof}

Let $\D$ be any fusion subcategory of $\C$. 
Define a pairing
$$
  B_{\D} : K_{\D}\times H_{\D}\rightarrow k^{\times}
$$
as follows. Let $(a,\chi)\in \D\cap\Gamma$, and
\begin{equation}\label{BD-defn}
   B_{\D}(x^{-1}ax,h):= \frac{\beta_a(x,h)\beta_a(xh,x^{-1})}
       {\beta_a(x,x^{-1})} \frac{\chi(xhx^{-1})}{\deg \chi}
\end{equation}
for all $x\in G$, $h\in H_{\D}$.
We will show that $B_{\D}$ is well-defined and does not depend on $\chi$.
For this, we give an equivalent definition: Write $h=y^{-1}by$ for
$b\in H_{\D}\cap R$, $y\in G$, $(b,\chi')\in \D'\cap \Gamma$.
Since $(a,\chi)$ and $(b,\chi')$ centralize each other, we have
by the definition (\ref{BD-defn}) of $B_{\D}$ and Lemma
\ref{lem-twisted}(ii) that
\begin{equation}\label{BD-alternate}
  B_{\D}(x^{-1}ax,y^{-1}by) = \frac{\beta_b(y,y^{-1})}
   {\beta_b(y,x^{-1}ax)\beta_b(yx^{-1}ax,y^{-1})}
   \frac{\deg\chi'}{\chi'(yx^{-1}axy^{-1})}.
\end{equation}
This proves that $B_{\D}$ does not depend on the choice of $\chi$.
Now suppose $z^{-1}az=x^{-1}ax$ for some $z\in G$.
Then by (\ref{BD-alternate}),
\begin{eqnarray*}
   B_{\D}(z^{-1}az,y^{-1}by) & = & \frac{\beta_b(y,y^{-1})}
       {\beta_b(y,z^{-1}az)\beta_b(yz^{-1}az,y^{-1})}
        \frac{\deg\chi'}{\chi'(yz^{-1}azy^{-1})}\\
       &=&  \frac{\beta_b(y,y^{-1})}
       {\beta_b(y,x^{-1}ax)\beta_b(yx^{-1}ax,y^{-1})}
        \frac{\deg\chi'}{\chi'(yx^{-1}axy^{-1})}\\
     &=& B_{\D}(x^{-1}ax,y^{-1}by).
\end{eqnarray*}
Therefore $B_{\D}$ is well-defined.

Next we give a definition of a $G$-invariant $\omega$-bicharacter, 
generalizing \cite[Defns.\ 4.5, 4.6]{NN} (the case $K=H$),
and show that $B_{\D}$ satisfies the definition. 

\begin{definition}\label{inv-bichar}
Let $H,K$ be normal subgroups of $G$ that centralize each other, 
and let $B:K\times H\rightarrow k^{\times}$ be a function. 

We say that $B$ is an $\omega$-{\em bicharacter} if
\begin{eqnarray*}
  \mbox{(i) }B(x,yz) &=& \beta_x^{-1}(y,z) B(x,y)B(x,z) \ \mbox{ and}\\
  \mbox{(ii) }B(wx,y) &=& \beta_y(w,x) B(w,y)B(x,y)
\end{eqnarray*}
for all $w,x\in K$, $y,z\in H$.
Equivalently, $B(x, \cdot)$ is a $\beta_x$-character of $H$ and
$B(-,y)$ is a $\beta_y^{-1}$-character of $K$.

We say that $B$ is $G$-{\em invariant} if 
$$
   B(x^{-1}kx,h) = \frac{\beta_k(x,h)\beta_k(xh,x^{-1})}{\beta_k(x,x^{-1})}
                       B(k,xhx^{-1})
$$
for all $x,y\in G$, $h\in H$, $k\in K$.
\end{definition}

\begin{remark}\label{rem:G-inv}
\begin{enumerate}
\item[(i)] The $G$-invariance property corresponds to the following action
of $G$: Let $\{\delta_x\}_{x\in G}$ be the basis of the linear dual $(kG)^*$
of the group algebra, that is dual to $G$, that is $\delta_x(y)=\delta_{x,y}$
for all $x,y\in G$.
Then we may identify a basis of $D^{\omega}(G)$ with the set 
$\{\delta_x\overline{g}\}_{x,g\in G}$. In this notation, the subalgebra
$k\delta_xC_G(x)$ of $D^{\omega}(G)$ is isomorphic to the twisted group algebra
$k^{\beta_x}[C_G(x)]$. A group element $g\in G$ acts on $D^{\omega}(G)$ via
conjugation by the invertible element $\overline{g} = \sum_{x\in G}\delta_x
\overline{g}$, and takes the subalgebra $k^{\beta_x}[C_G(x)]$ to
$k^{\beta_{g^{-1}xg}}[C_G(g^{-1}xg)]$. Under these identifications,
$B$ is $G$-invariant if and only if the function $B(k, \cdot)$ on 
$k^{\beta_k}[H]$ is taken to the function $B(x^{-1}kx,\cdot)$ on 
$k^{\beta_{x^{-1}kx}}[H]$ via conjugation by $\overline{x}$.

\item[(ii)] There is a symmetry in the definition of an $\omega$-bicharacter
that appears to be lacking in the definition of $G$-invariance.
However, a consequence of the next lemma is that
there is an equivalent definition of
$G$-invariance, considering $B$ instead to be a function in the first argument,
fixing the second.
\end{enumerate} 
\end{remark}

\begin{lemma}
\label{lemma:nu-beta-reln}
Let $x,h,k\in G$ for which $hk=kh$. Then
\begin{itemize}
 \item[(i)] $\displaystyle{\frac{\nu_x(h,k)}{\nu_x(k,h)}
= \frac{\beta_{xhx^{-1}}(x, x^{-1})}
{\beta_{xhx^{-1}}(x,k) \beta_{xhx^{-1}}(xk, x^{-1})}}$,
 \item[(ii)] $\displaystyle{\frac{\nu_x(x^{-1}hx, x^{-1}kx)}  
     {\nu_x(x^{-1}kx, x^{-1}hx)}= \frac{\nu_{x^{-1}}(k,h)}  {\nu_{x^{-1}}(h,k)}}$, and
\item[(iii)] $\displaystyle{\frac{\beta_k(y^{-1},y)}
{\beta_k(y^{-1},h)\beta_k(y^{-1}h,y)} =
\frac{\beta_{h}(y,y^{-1})} 
   {\beta_{h}(y,k)\beta_{h}(yk,y^{-1})}  }$.
\end{itemize}
\end{lemma}

\begin{proof}
(i) We must show that 
$$ \frac{\nu_x(h,k)}{\nu_x(k,h)} \cdot  
\frac{\beta_{xhx^{-1}}(x,k) \beta_{xhx^{-1}}(xk, x^{-1})}
{\beta_{xhx^{-1}}(x, x^{-1})} = 1. $$
This follows by applying the definitions \eqref{nu-defn} and \eqref{beta-defn},
and the $3$-cocycle condition \eqref{3-cocycle condition} to the
tuples $(xkx^{-1},x,h,x^{-1})$, $(xkx^{-1}, xhx^{-1}, x, x^{-1})$, 
$(xhx^{-1}, xkx^{-1},x,x^{-1})$, $(xkx^{-1},x,x^{-1},xhx^{-1})$.

(ii) Applying \eqref{nu-eta-reln} to the quadruples 
$(x,x^{-1},h,k)$ and $(x,x^{-1},k,h)$ we obtain:
\begin{eqnarray*}
\nu_x(x^{-1}hx, x^{-1}kx) \nu_{x^{-1}}(h,k) &=& 
\frac{\eta_{hk}(x, x^{-1})}{\eta_h(x, x^{-1}) \eta_k(x, x^{-1})},\\
\nu_x(x^{-1}kx, x^{-1}hx) \nu_{x^{-1}}(k,h) &=& 
\frac{\eta_{kh}(x, x^{-1})}{\eta_k(x, x^{-1}) \eta_h(x, x^{-1})}.
\end{eqnarray*}
Since $hk=kh$ the stated identity follows. 

(iii)
Applying the identity in part (i) to the triples
$(y^{-1}, yky^{-1}, h)$ and $(y, y^{-1}hy,k)$, the claimed identity
is equivalent to
$$ 
\frac{\nu_{y^{-1}}(yky^{-1}, h)\nu_y(k,y^{-1}hy)}
   {\nu_{y^{-1}}(h, yky^{-1})\nu_y(y^{-1}hy,k)} = 1,
$$
which is seen to be true by applying the identity in part (ii)
to the triple $(y, yky^{-1}, h)$.
\end{proof}

It follows from part (iii) of the lemma that a function 
$B:K\times H\rightarrow k^{\times}$ 
is $G$-invariant if and only if
$$
 B(k,y^{-1}hy) = \frac{\beta_h(y,y^{-1})}{\beta_h(y,k)\beta_h(yk,y^{-1})}
                         B(yky^{-1},h)
$$
for all $x,y\in G$, $h\in H$, $k\in K$.

\begin{proposition}
Let $\D$ be a fusion subcategory of $\C$.
Then $B_{\D}$ is a $G$-invariant $\omega$-bicharacter on
$K_{\D}\times H_{\D}$.
\end{proposition}

\begin{proof}
First we check $G$-invariance.
By definition (\ref{BD-defn}), for all $a\in K\cap R$, we have
$$
  B_{\D}(x^{-1}ax,h) = \frac{\beta_a(x,h)\beta_a(xh,x^{-1})}
    {\beta_a(x,x^{-1})} \frac{\chi(xhx^{-1})}{\deg\chi}
  =  \frac{\beta_a(x,h)\beta_a(xh,x^{-1})}
    {\beta_a(x,x^{-1})} B_{\D}(a,xhx^{-1}),
$$
the second equality by (\ref{BD-defn}) with $h$ replaced by $xhx^{-1}$
and $x$ replaced by the identity element.
More generally, let $k=y^{-1}ay$, and apply the above equation twice, the first
time to express $B((yx)^{-1}ayx, h)$ as a scalar multiple of $B(a, (yx)h(yx)^{-1})$,
and the second time to express $B(a, y(xhx^{-1})y^{-1})$ as a scalar multiple
of $B(y^{-1}ay, xhx^{-1})$.
We obtain in this way:
\begin{equation}\label{eqn:more-invariance}
  B(x^{-1}kx,h) = \frac{\beta_a(y,y^{-1})\beta_a(yx,h)\beta_a(yxh,x^{-1}
   y^{-1})}{\beta_a(y,xhx^{-1})\beta_a(yxhx^{-1}, y^{-1})\beta_a(yx,x^{-1}y^{-1})}
   B(k,xhx^{-1}).
\end{equation}
This is equivalent to $G$-invariance:
Apply (\ref{beta-reln}) thrice to 
$$
   \frac{\beta_{y^{-1}ay}(x,h)\beta_{y^{-1}ay}(xh,x^{-1})}
     {\beta_{y^{-1}ay}(x,x^{-1})}
$$
to express all scalars as images of $\beta_a$,
and then apply (\ref{beta-reln}) to the quadruples
$(a, yxh, x^{-1}, y^{-1})$ and $(a, yx, x^{-1}, y^{-1})$ to obtain
the expression \eqref{eqn:more-invariance}. 

Now let $h_1,h_2\in H$ and $a\in K\cap R$.
By the definition of $B_{\D}$, we have
$$
  B_{\D}(a,h_1)B_{\D}(a,h_2) = \frac{\chi(h_1)}{\deg\chi}
   \frac{\chi(h_2)}{\deg\chi} = \beta_a(h_1,h_2)\frac{\chi(h_1h_2)}
     {\deg\chi} = \beta_a(h_1,h_2)B_{\D}(a,h_1h_2).
$$
Let $x\in G$. By $G$-invariance of $B_{\D}$ we now have that
$B_{\D}(x^{-1}ax,h_1)B_{\D}(x^{-1}ax,h_2)$ is equal to
\begin{eqnarray*}
&&\frac{\beta_a(x,h_1)\beta_a(xh_1,x^{-1})\beta_a(x,h_2)\beta_a(xh_2,x^{-1})}
    {\beta_a(x,x^{-1})^2} B(a,xh_1x^{-1})B(a,xh_2x^{-1})\\
&&= \ \frac{\beta_a(xh_1x^{-1},xh_2x^{-1})\beta_a(x,h_1)\beta_a(xh_1,x^{-1})
       \beta_a(x,h_2)\beta_a(xh_2,x^{-1})}{\beta_a(x,x^{-1})^2}
       B(a,xh_2h_2x^{-1}).
\end{eqnarray*}
On the other hand, 
$$
  \beta_{x^{-1}ax}(h_1,h_2)B_{\D}(x^{-1}ax,h_1h_2) = 
   \frac{\beta_{x^{-1}ax}(h_1,h_2)\beta_a(x,h_1h_2)\beta_a(xh_1h_2,x^{-1})}
      {\beta_a(x,x^{-1})} B(a,xh_1h_2x^{-1}).
$$
Comparing, the two will be the same if and only if
$$
  \frac{\beta_a(xh_1x^{-1},xh_2x^{-1})\beta_a(x,h_1)\beta_a(xh_1,x^{-1})
   \beta_a(x,h_2)\beta_a(xh_2,x^{-1})}
    {\beta_a(x,x^{-1})\beta_{x^{-1}ax}(h_1,h_2)\beta_a(x,h_1h_2)
  \beta_a(xh_1h_2,x^{-1})} = 1.
$$
This we prove just as in the proof of \cite[Lemma 4.10]{NN}:
We apply (\ref{beta-reln}) for $\beta_a$
successively to the five triples $(x,h_1,h_2)$,  
$(x, h_2,x^{-1})$,  $(xh_1,x^{-1},xh_2x^{-1})$, $(xh_1,h_2,x^{-1})$,  
$(x,x^{-1},xh_2x^{-1})$ and make a corresponding substitution each time.
It follows that
$B_{\D}$ satisfies the first equation in the definition of
$\omega$-bicharacter.
To see that it satisfies the second, we use the alternative definition
(\ref{BD-alternate}) of $B_{\D}$, and apply an analogous argument.
\end{proof}

We have associated, to each fusion subcategory $\D$ of $\C$, a 
canonical triple $(K_{\D},H_{\D},B_{\D})$, where $K_{\D}$ and $H_{\D}$
are normal subgroups of $G$ centralizing each other, and $B_{\D}$ is
a $G$-invariant $\omega$-bicharacter.
Next we associate to each triple a fusion subcategory.

\subsection{Construction of a fusion subcategory of 
$\Rep(D^{\omega}(G))$}

Let $K,H$ be normal subgroups of $G$ that centralize each other,
and $B:K\times H\rightarrow k^{\times}$ a $G$-invariant 
$\omega$-bicharacter.
Define 
\begin{equation}
\label{definition of twisted S(K,H,B)}
\begin{split}
\S{(K,H,B)} := 
&\text{ full Abelian subcategory of } \C  \text{ generated by } \\
&\left\{(a, \, \chi) \in \Gamma \,\, \vline \,
\begin{tabular}{l}
$a \in K \cap R \text{ and } \chi \text{ is an irreducible }
  \beta_a\text{-character of } C_G(a)$ \\
$\text{ such that } \chi(h) = B(a, \, h) \,\deg \chi, \text{ for all } h \in H$
\end{tabular}
\right\}.
\end{split}
\end{equation}
We will prove that $\S(K,H,B)$ is a fusion subcategory of $\C$ and
determine its centralizer.
We will need \cite[Lemma 4.4]{NN}, proved by Clifford theory
\cite[Theorem 7.8.1]{Ka} and Frobenius reciprocity 
\cite[Proposition 7.4.8]{Ka} for projective characters:

\begin{lemma}\label{lem-KE}
Let $E$ be a normal subgroup of a finite group $F$, and $\alpha$ a 2-cocycle on $F$.
Let $\Irr(F)$ denote the set of irreducible $\alpha$-projective characters
of $F$. Let $\rho$ be an $F$-invariant $\alpha|_{E\times E}$-projective
character of $E$ of degree 1.
Then 
$$
  \sum_{\chi\in\Irr(F) : \chi|_E = (\deg\chi)\rho}
   (\deg\chi)^2 = [F:E].
$$
\end{lemma}

To apply the lemma to our situation, replace $E$ by $H$ and $F$
by $C_G(a)$, and note that by $G$-invariance of $B$, the function
$B(a, - )$ is a $C_G(a)$-invariant $\beta_a|_{H\times H}$-projective
character of $H$ of degree 1. 
One consequence of the lemma then is that there exists at least
one $\chi\in\Irr(C_G(a))$ such that $(a,\chi)\in\S(K,H,B)$.
This implies in particular that $K$ is the support of $\S(K,H,B)$.

\begin{lemma}\label{lem-dimS}
The dimension of $\S(K,H,B)$ is $|K| [G:H]$.
\end{lemma}

\begin{proof}
The proof is the same as for $\omega=1$  (Lemma \ref{dim S(K,H,B)}), replacing
Lemma \ref{FR} by Lemma~\ref{lem-KE}.
\end{proof}

Let $(B^{\op})^{-1}: H\times K\rightarrow k^{\times}$ be the
function defined by $(B^{\op})^{-1}(h,k) = B(k,h)^{-1}$.
Then $(B^{\op})^{-1}$ is a $G$-invariant $\omega$-bicharacter,
by symmetry of Definition \ref{inv-bichar} (see Remark \ref{rem:G-inv}(ii)).

\begin{lemma}\label{lemma:twisted-prime}
$\S(K,H,B)$ is a fusion subcategory of $\C$ and 
$\S(K,H,B) ' = \S(H,K,(B^{\op})^{-1})$.
\end{lemma}

\begin{proof}
First we show that $\S(K,H,B)$ and $\S(H,K,(B^{\op})^{-1})$
centralize each other.
Let $(a,\chi)\in \S(K,H,B)\cap\Gamma$ and $(b,\chi')\in\S(H,K,
(B^{\op})^{-1})\cap\Gamma$.
Then the conjugacy classes $K_a$ and $K_b$ centralize each other.
We check Lemma \ref{lem-twisted}(ii), using the two
equations (\ref{BD-defn}) and (\ref{BD-alternate}):
\begin{eqnarray*}
&&\frac{\beta_a(x,y^{-1}by)\beta_a(xy^{-1}by,x^{-1})\beta_b
   (y,x^{-1}ax)\beta_b(yx^{-1}ax,y^{-1})}{\beta_a(x,x^{-1})
   \beta_b(y,y^{-1})} \chi(xy^{-1}byx^{-1})\chi'(yx^{-1}axy^{-1})\\
    && = \ \frac{\beta_a(x,y^{-1}by)\beta_a(xy^{-1}by,x^{-1})\beta_b
   (y,x^{-1}ax)\beta_b(yx^{-1}ax,y^{-1})}{\beta_a(x,x^{-1})
   \beta_b(y,y^{-1})} \\
 &&\hspace{2cm}\cdot B(a,xy^{-1}byx^{-1})\deg\chi (B^{\op})^{-1}(b,yx^{-1}
           axy^{-1})\deg\chi' .
\end{eqnarray*}
By $G$-invariance, this may be rewritten as
$$
   B(x^{-1}ax,y^{-1}by)\deg\chi B^{-1}(x^{-1}ax,y^{-1}by)
  \deg\chi' = \deg\chi\deg\chi',
$$
as desired.
Therefore $\S(K,H,B)$ and $\S(H,K, (B^{\op})^{-1})$ centralize
each other.

By Lemma \ref{lem-dimS},
$$
  \dim(\S(K,H,B))\dim(\S(H,K,(B^{\op})^{-1}) = |G|^2 = \dim \C.
$$
By Lemma \ref{abelian subcat. is fusion}, $\S(K,H,B)$ is a
fusion subcategory of $\C$ and $\S(K,H,B)'=\S(H,K,(B^{\op})^{-1})$.
\end{proof}

\begin{theorem}\label{theorem:twisted-main}
The assignments $\D\mapsto (K_{\D},H_{\D},B_{\D})$ and
$(K,H,B)\mapsto \S(K,H,B)$ are inverses of each other.
Thus there is a bijection between the set of fusion subcategories
of $\C$ and triples $(K,H,B)$ where $K,H$ are normal subgroups
of $G$ that centralize each other and $B: K\times H \mapsto k^{\times}$
is a $G$-invariant $\omega$-bicharacter.
\end{theorem}

\begin{proof}
First we show that $(K_{\S(K,H,B)}, H_{\S(K,H,B)}, B_{\S(K,H,B)})
=(K,H,B)$. We have already observed that $K$ is the support of
$\S(K,H,B)$, so $K_{\S(K,H,B)}=K$.
It follows that $H_{\S(K,H,B)}$, which is the support of
$\S(K,H,B)'=\S(H,K,(B^{\op})^{-1})$, is $H$.
Let $a\in K\cap R$, $(a,\chi)\in\S(K,H,B)\cap\Gamma$, and $h\in H$.
Then by the definition of $B_{\S(K,H,B)}$, we have
$$
  B_{\S(K,H,B)}(a,h)=\frac{\chi(h)}{\deg\chi}.
$$
On the other hand, since $(a,\chi)\in\S(K,H,B)$, we have 
$\chi(h) = B(a,h)\deg\chi$, so $B(a,h)=B_{\S(K,H,B)}(a,h)$.
By $G$-invariance, this implies $B=B_{\S(K,H,B)}$.

Now we show that $\S(K_{\D},H_{\D},B_{\D})=\D$.
Let $(a,\chi)\in \S(K_{\D},H_{\D},B_{\D})\cap\Gamma$.
Then for all $h\in H_{\D}$, by the definitions of $\S(K_{\D},H_{\D},
B_{\D})$ and of $B_{\D}$, respectively, we have
$$
  \frac{\chi(h)}{\deg\chi} = B_{\D}(a,h) = \frac{\chi'(h)}
   {\deg\chi'}
$$
for any $\chi'$ such that $(a,\chi')$ is simple in $\D$.
By Proposition \ref{prop-HDKD}, we have $K_{\D'}=H_{\D}$, so the
above equation holds for all $h\in K_{\D'}$.
By Lemma \ref{lem-E}, we have $(a,\chi)\in\D$.
We have shown $\S(K_{\D},H_{\D},B_{\D})\subseteq \D$.
Similarly $\S(K_{\D'},H_{\D'},B_{\D'}) = \S(H_{\D},K_{\D},B_{\D'})\subseteq\D'$.
By Lemma \ref{lem-dimS},
$$
  \dim(\S(K_{\D},H_{\D},B_{\D}))\dim(\S(H_{\D},K_{\D},B_{\D'}))
   =|G|^2 = \dim\C.
$$
This forces $\S(H_{\D},K_{\D},B_{\D'}) = \S(K_{\D},H_{\D},B_{\D})'$, so
that $\D'\subseteq \S(H_{\D},K_{\D},B_{\D'})$.
Therefore $\D' = \S(H_{\D},K_{\D},B_{\D'})$, and by symmetry,
$\D = \S(K_{\D},H_{\D},B_{\D})$.
\end{proof}

\begin{remark}\label{remark:braided}
Let $\mathcal C$ be any braided group-theoretical fusion category
(see Section \ref{fusion cats.}).
The braiding yields a canonical embedding of $\mathcal C$ into its 
center $\Z({\mathcal C})$.
(See \cite{K} for the center construction.)
Since $\mathcal C$ is group-theoretical, $\Z({\mathcal C})$ is equivalent 
to $\Z(\Vec^{\omega}_G)\cong \Rep(D^{\omega}(G))$ for some finite
group $G$ and 3-cocycle $\omega$. (This follows from 
\cite[Sect.\ 8.8]{ENO}.
See also the proof of \cite[Thm.\ 1.2]{Na} and \cite[Prop.\ 3.3]{Mj}.)
Therefore $\mathcal C$ may be realized as a fusion subcategory of
$\Rep(D^{\omega}(G))$.
By Theorem \ref{theorem:twisted-main}, $\mathcal C$ is equivalent
to a category of the form ${\mathcal S}(K,H,B)$ defined in
(\ref{definition of twisted S(K,H,B)}) for some normal subgroups
$K,H$ of $G$ that centralize each other and $G$-invariant 
$\omega$-bicharacter $B$ on $K\times H$.
\end{remark}

\end{section}
%%%%%%%%%%%%%%%%%%%%%%%%%%%%%%%%%%%%%%%%%%%%%%%%%%%%%%%%%%%%%%%%%

%%%%%%%%%%%%%%%%%%%%%%%%%%%%%%%%%%%%%%%%%%%%%%%%%%%%%%%%%%%%%%%%%
\begin{section}
{Some invariants and the lattice of fusion subcategories of
$\Rep(D^{\omega}(G))$}
\label{sec:lattice}

%%%%%%%%%%%%%%%%%%%%%%%%%%%%%%%%%%%%%%%%%%%%%%%%%%%%%%%%%%%%%%%%%
\begin{subsection}{The lattice}

Let $K,H,K', H'$ be normal subgroups of $G$ such that $K$ and $H$
centralize each other, and $K'$ and $H'$ centralize each other.
Let $B:K\times H\rightarrow k^{\times}$ and $B': K'\times H'\rightarrow
k^{\times}$ be $G$-invariant $\omega$-bicharacters.

\begin{proposition}
\label{prop:subcat prime}
$\S(K,H,B) \subseteq \S(K',H',B')$ if and only if 
$K \subseteq K', \, H' \subseteq H$, and $B|_{K \times H'}
= B'|_{K \times H'}$.
\end{proposition}
\begin{proof}
The statement follows from the definitions since $K$ is the support of 
$S(K,H,B)$ by Lemma \ref{lem-KE}.
\end{proof}

Let $\varphi_{B,B'}: K\cap K'\rightarrow\widehat{H\cap H'}$ be the function
defined by 
$$
  \varphi_{B,B'}(a) = B(a,\cdot)|_{H\cap H'}(B')^{-1}(a,\cdot)|_{H\cap H'},
$$
where $\widehat{H\cap H'}$ denotes the group of group homomorphisms
from $H\cap H'$ to $k^{\times}$.
It is straightforward to check that for each $a\in K\cap K'$, 
$\varphi_{B,B'}(a)$ is indeed a group homomorphism from $H\cap H'$ to $k^{\times}$,
and that $\varphi_{B,B'}$ is itself a group homomorphism from $K\cap K'$
to $\widehat{H\cap H'}$. Let $\psi_{B,B'}: \ker\varphi_{B,B'}\times HH'
\rightarrow k^{\times}$ be the function defined by 
$$
  \psi_{B,B'}(a,hh') = \beta_a^{-1}(h,h')B(a,h)B'(a,h')
$$
for all $a\in K$, $h\in H$, $h'\in H'$.
Note that $\psi_{B,B'}$ is well-defined: 
Suppose $h_1h_1' = h_2h_2'$ for elements $h_1,h_2\in H$,
$h_1',h_2'\in H'$. 
We must check that
$$
  \beta_a^{-1}(h_1,h_1')B(a,h_1)B'(a,h_1') =
 \beta_a^{-1}(h_2,h_2')B(a,h_2)B'(a,h_2'),
$$
or equivalently,
$$
  \beta_a^{-1}(h_1,h_1')B(a,h_2)^{-1}B(a,h_1) =
  \beta_a^{-1}(h_2,h_2')B'(a,h_2')B'(a,h_1')^{-1}.
$$
Since $B,B'$ are $\omega$-bicharacters, we have
$$
  B(a,h_2)B(a,h_2^{-1})=\beta_a(h_2,h_2^{-1})B(a,1)=\beta_a(h_2,h_2^{-1}),
$$
so $B(a,h_2)^{-1}=\beta_a^{-1}(h_2,h_2^{-1})B(a,h_2^{-1})$, and similarly,
$B'(a,h_1')^{-1}=\beta_a^{-1}(h_1',(h_1')^{-1})B(a,(h_1')^{-1})$.
Thus we must show that
$$
 \frac{\beta_a(h_2^{-1},h_1)}{\beta_a(h_1,h_1')\beta_a(h_2,h_2^{-1})}
    B(a,h_2^{-1}h_1) = \frac{\beta_a(h_2',(h_1')^{-1})}{\beta_a(h_2,h_2')
   \beta_a(h_1',(h_1')^{-1})} B'(a,h_2'(h_1')^{-1}).
$$
Now, $h_2^{-1}h_1=h_2'(h_1')^{-1} \in H\cap H'$, and since $a\in\ker\varphi_{B,B'}$,
we have $B(a,h_2^{-1}h_1)=B'(a,h_2'(h_1')^{-1})$. Thus we must show that
$$
  1 = \frac{\beta_a(h_1,h_1')\beta_a(h_2,h_2^{-1})\beta_a(h_2',(h_1')^{-1})}
   {\beta_a(h_2^{-1},h_1)\beta_a(h_2,h_2')\beta_a(h_1',(h_1')^{-1})}.
$$
This follows by applying the 2-cocycle identity for $\beta_a$ successively
to the triples $(h_2^{-1},h_1,h_1')$, $(h_2^{-1},h_2,h_2')$,
and $(h_2',(h_1')^{-1},h_1')$.
Therefore $\psi_{B,B'}$ is well-defined.

%\begin{lemma}
%\label{beta-reln 1}
%Let $a,x,y,z \in G$, and assume that $y$ centralizes $x^{-1}ax$. Then
%$$
%  \frac{\beta_a(xyx^{-1},xzx^{-1})\beta_a(x,y)\beta_a(xy,x^{-1})
%   \beta_a(x,z)\beta_a(xz,x^{-1})} 
%   {\beta_{x^{-1}ax}(y,z) \beta_a(x,yz)\beta_a(xyz,x^{-1})\beta_a(x,x^{-1})}= 1.
%$$
%\end{lemma}
%
%
%\begin{proof}
%This follows by applying the identity \eqref{beta-reln} successively 
%to the quadruples $(a,x,z,x^{-1}),$ $(a,xy,x^{-1},xzx^{-1})$, $(a,xy,z,x^{-1}), \,
%(a,x,x^{-1},xzx^{-1}), \text{ and } (a,x,y,z)$.
%\end{proof}

\begin{lemma}
The function $\psi_{B,B'}$ is a  $G$-invariant
$\omega$-bicharacter on $\ker\varphi_{B,B'}\times HH'$.
\end{lemma}

\begin{proof}
We first prove that identity (ii) in Definition \ref{inv-bichar} holds
for $\psi_{B,B'}$:
Let $a,a'\in\ker\varphi_{B,B'}$, and $h\in H$, $h'\in H'$.
We must prove
$$
  \psi_{B,B'}(aa',hh')=\beta_{hh'}(a,a')\psi_{B,B'}(a,hh') 
   \psi_{B,B'}(a',hh').
$$
The left side is 
$$
  \frac{1}{\beta_{aa'}(h,h')}B(aa',h)B'(aa',h')=
  \frac{\beta_h(a,a')\beta_{h'}(a,a')}{\beta_{aa'}(h,h')}
   B(a,h)B(a',h)B'(a,h')B'(a',h').
$$
The right side is 
$$
   \frac{\beta_{hh'}(a,a')}{\beta_a(h,h')\beta_{a'}(h,h')}
  B(a,h)B'(a,h')B(a',h)B'(a',h').
$$
The two are equal if and only if
$$
  1 = \frac{\beta_h(a,a')\beta_{h'}(a,a')\beta_a(h,h')\beta_{a'}(h,h')}
    {\beta_{aa'}(h,h')\beta_{hh'}(a,a')}.
$$
This is true  by the identity
\eqref{gamma beta relation 1} and by \eqref{four-relns} (note
$a,a'$ centralize $h,h'$).

Next we prove that $\psi_{B,B'}$ is $G$-invariant:
We must prove that for all $x\in G$,
$$
  \psi_{B,B'}(x^{-1}ax,hh')= \frac{\beta_a(x,hh')\beta_a(xhh',x^{-1})}
   {\beta_a(x,x^{-1})} \psi_{B,B'}(a,xhh'x^{-1}).
$$
By the definition of $\psi_{B, B'}$ and $G$-invariance of $B,B'$,
the left side is
\begin{eqnarray*}
 && \frac{1}{\beta_{x^{-1}ax}(h,h')}B(x^{-1}ax,h)B'(x^{-1}ax,h')\\
 &&\hspace{.5cm}  
 =\frac{\beta_a(x,h)\beta_a(xh,x^{-1})\beta_a(x,h')\beta_a(xh',x^{-1})}
  {\beta_{x^{-1}ax}(h,h')\beta_a(x,x^{-1})^2}
   B(a,xhx^{-1})B'(a,xh'x^{-1}).
\end{eqnarray*}
The right side is
$$
  \frac{\beta_a(x,hh')\beta_a(xhh',x^{-1})}{\beta_a(x,x^{-1})\beta_a(xhx^{-1},
    xh'x^{-1})} B(a,xhx^{-1})B'(a,xh'x^{-1}).
$$
So the left side equals the right side if and only if
$$
  1=\frac{\beta_a(x,h)\beta_a(xh,x^{-1})\beta_a(x,h')\beta_a(xh',x^{-1})
   \beta_a(xhx^{-1},xh'x^{-1})} {\beta_{x^{-1}ax}(h,h')\beta_a(x,x^{-1})
   \beta_a(x,hh')\beta_a(xhh',x^{-1})}.
$$
This follows by applying the identity (\ref{beta-reln}) to the
tuples $(a,x,h,h')$, $(a,xh,x^{-1},xh'x^{-1})$, $(a,x,x^{-1},xh'x^{-1})$,
$(a,xh,h',x^{-1})$, and $(a,x,h',x^{-1})$.

Finally we prove that identity (i) in Definition \ref{inv-bichar} holds
for $\psi_{B,B'}$. 
Let $a\in\ker\varphi_{B,B'}$ and $y_1,z_1\in H$, $y_2,z_2\in H'$,
$y=y_1y_2$, $z=z_1z_2$. 
We show that
$$
   \psi_{B,B'}(a,yz) =\beta_a^{-1}(y,z)\psi_{B,B'}(a,y)\psi_{B,B'}
   (a,z).
$$
Since $B'$ is $G$-invariant and $z_1\in C_G(a)$, the left side 
$\psi_{B,B'}(a,(y_1z_1)(z_1^{-1}y_2z_1z_2))$ is equal to 
\begin{eqnarray*}
&=&
   \beta_a^{-1}(y_1z_1, z_1^{-1}y_2z_1z_2) B(a,y_1z_1)B'(a,z_1^{-1}y_2z_1z_2)\\
   &=&\beta_a^{-1}(y_1z_1,z_1^{-1}y_2z_1z_2)\beta_a^{-1}(y_1,z_1)\beta_a^{-1}
  (z_1^{-1}y_2z_1,z_2) B(a,y_1)B(a,z_1)B'(a,z_1^{-1}y_2z_1)B'(a,z_2)\\
&=& \frac{\beta_a(z_1^{-1},z_1)}{\beta_a(y_1z_1, z_1^{-1}y_2z_1z_2)\beta_a(y_1,z_1)
  \beta_a(z_1^{-1}y_2z_1,z_2) \beta_a(z_1^{-1},y_2)\beta_a(z_1^{-1}y_2,z_1)}
  \\
  &&\hspace{4cm}\cdot B(a,y_1)B(a,z_1)B'(a,y_2)B'(a,z_2).
\end{eqnarray*}
The right side is
$$
 \beta_a^{-1}(y_1y_2,z_1z_2)\beta_a^{-1}(y_1,y_2)\beta_a^{-1}(z_1,z_2)
  B(a,y_1)B'(a,y_2)B(a,z_1)'(a,z_2).
$$
The two are equal if and only if
$$
  1 = \frac{\beta_a(z_1^{-1},z_1)\beta_a(y_1,y_2)\beta_a(z_1,z_2)
   \beta_a(y_1y_2,z_1z_2)} {\beta_a(y_1z_1, z_1^{-1}y_2z_1z_2)
  \beta_a(y_1,z_1)\beta_a(z_1^{-1}y_2z_1, z_2)\beta_a(z_1^{-1},y_2)
  \beta_a(z_1^{-1}y_2,z_1)}.
$$
This follows from the identity $\beta_a(z_1^{-1},z_1)=\beta_a(z_1,z_1^{-1})$
(obtained by applying the 2-cocycle identity to $(z_1^{-1},z_1,z_1^{-1})$)
and application of the 2-cocycle identity for $\beta_a$ to the triples
$(y_1,z_1,z_1^{-1})$, $(z_1^{-1},y_2,z_1)$, $(z_1^{-1},y_2z_1,z_2)$,
$(y_2,z_1,z_2)$, $(y_1z_1,z_1^{-1}y_2,z_1z_2)$, and $(y_1z_1,z_1^{-1},y_2)$.
\end{proof}

\begin{proposition}\label{prop:intersection and join}
\begin{enumerate}
\item[(i)]
${\mathcal S}(K,H,B)\cap {\mathcal S}(K',H',B') = {\mathcal S}(\ker\varphi_{B,B'},
HH',\psi_{B,B'})$.
\item[(ii)] ${\mathcal S}(K,H,B)\vee {\mathcal S}(K',H',B') =
{\mathcal S}(KK', \ker\varphi_{B^{\op}, (B')^{\op}}, (\psi_{B^{\op},(B')^{\op}})^{\op}))$.
\end{enumerate}
\end{proposition}

\begin{proof}
The proof is the same as for $\omega=1$ (Proposition \ref{intersection and join}),
using Lemma \ref{lemma:twisted-prime} in place of Lemma \ref{centralizer of S(K,H,B)}.
\end{proof}

We will need the following lemma in order to characterize isotropic subcategories.

\begin{lemma}
\label{lemma:noname prime}
Let $B: K \times H \to k^\times$ be a $G$-invariant $\omega$-bicharacter with
$K \subseteq H$ and $B(a,a)=1$, for all $a \in K \cap R$. Then
$B(k,k)=1$, for all $k \in K$.
\end{lemma}
\begin{proof}
Let $k \in K$ and write $k=x^{-1}ax, \, x \in G, \, a \in K \cap R$. 
By $G$-invariance of $B$ 
and the equality $B(a,a)=1$ we obtain
$$ B(x^{-1}ax, x^{-1}ax) = \frac{\beta_a(x,x^{-1}ax) \beta_a(ax, x^{-1})}
{\beta_a(x,x^{-1})}. $$
Applying the definition \eqref{beta-defn} of $\beta_a$ and the $3$-cocycle
condition \eqref{3-cocycle condition} for $\omega$ to the quadruples
$(a,x,x^{-1}ax,x^{-1})$ and $(a,x,x^{-1},a)$, the right side
of the above equality is equal to $1$.
\end{proof}

We next characterize symmetric, isotropic and Lagrangian subcategories.
Definitions were given in Section \ref{sec:lattices}.

\begin{proposition}
The fusion subcategory $\S(K,H,B) \subseteq \Rep(D^{\omega}(G))$ is
\begin{enumerate}
\item[(i)] symmetric if and only if $K \subseteq H$ and
$B(k_1, k_2) B(k_2, k_1) =1$, for all $k_1, k_2 \in K$,
\item[(ii)] isotropic if and only if $K \subseteq H$ and
$B|_{K \times K}$ is alternating,
\item[(iii)] Lagrangian if and only $K=H$ and $B$ is alternating.
\end{enumerate}
\end{proposition}

\begin{proof}
The proof is the same as for $\omega=1$ (Proposition \ref{prop:sil}), 
using Lemma \ref{lemma:twisted-prime} and Proposition \ref{prop:subcat prime}
in place of Lemma \ref{centralizer of S(K,H,B)} and Proposition \ref{prop:subcat}.
We note that in part (ii) we also use Lemma \ref{lemma:noname prime}.
\end{proof}

\begin{proposition}\label{prop:z2}
\label{prop:Mueger center of S(K,H,B) twisted}
$\Z_2(\S(K,H,B))=
\S(\ker \varphi_{B,(B^{\op})^{-1}}, HK, \psi_{B,(B^{\op})^{-1}})$.
\end{proposition}
\begin{proof}
The proof is the same  as for $\omega=1$ (Proposition \ref{prop:Mueger center of S(K,H,B)}),
using Lemma \ref{lemma:twisted-prime} and Proposition \ref{prop:intersection and join}(i)
in place of Lemma \ref{centralizer of S(K,H,B)} and Proposition \ref{intersection and join}(i).
\end{proof}

Note that if $B$ is an $\omega$-bicharacter on $K\times H$, then $B^{\op}$
is an $\omega^{-1}$-bicharacter on $H\times K$.
Consequently, $BB^{\op}$ is a (symmetric) bicharacter on 
$(K\cap H)\times (K\cap H)$.

\begin{proposition}\label{thm:nondegenerate}
\begin{enumerate}
\item[(i)] The fusion subcategory ${\mathcal S}(K,H,B) \subseteq \Rep(D^{\omega}(G))$
is nondegenerate if and only if $HK=G$ and the symmetric bicharacter
$BB^{\op}|_{(K\cap H)\times (K\cap H)}$ is nondegenerate.
\item[(ii)] $\Rep(D^{\omega}(G))$ is prime if and only if there is no triple
$(K,H,B)$, where 
$K$ and $H$ are normal subgroups of $G$
that centralize each other, 
$(G,\{e\})\neq (K,H)\neq (\{e\},G)$,
$HK=G$, and $B$ is a $G$-invariant $\omega$-bicharacter 
on $K\times H$ such that $BB^{\op}|_{(K\cap H)\times (K\cap H)}$
is nondegenerate.
\end{enumerate}
\end{proposition}

\begin{proof}
The proof is the same as for $\omega=1$ (Proposition \ref{S(K,H,B) is nondegenerate iff}),
using Theorem \ref{theorem:twisted-main} and Proposition \ref{prop:z2} in
place of Theorem \ref{bijection} and Proposition \ref{prop:Mueger center of S(K,H,B)}.
\end{proof}

\end{subsection}
%%%%%%%%%%%%%%%%%%%%%%%%%%%%%%%%%%%%%%%%%%%%%%%%%%%%%%%%%%%%%%%%%

%%%%%%%%%%%%%%%%%%%%%%%%%%%%%%%%%%%%%%%%%%%%%%%%%%%%%%%%%%%%%%%%%
\begin{subsection}{The Gauss sum and central charge}

The definitions of Gauss sum and central charge of a premodular
category were recalled in Section \ref{gscc}.

\begin{proposition}
The Gauss sum of the fusion subcategory $\S(K,H,B)$ of $\Rep(D^{\omega}(G))$
is 
$$
\tau(\S(K,H,B)) = \frac{|G|}{|H|} 
\sum_{a \in K \cap H \cap R} |K_a| B(a, a).
$$ 
When $\S(K,H,B)$ is nondegenerate its Gauss sum is
$$
\tau(\S(K,H,B)) = \frac{|K|}{|K \cap H|} 
\sum_{a \in K \cap H} B(a, a).
$$
and its central charge is
$$
\zeta(\S(K,H,B)) = \frac{1}{\sqrt{|K \cap H|}} 
\sum_{a \in K \cap H} B(a, a).
$$  
\end{proposition}

\begin{proof}
The proof is the same as for $\omega=1$ (Proposition \ref{prop:gauss sum}),
using Proposition \ref{thm:nondegenerate} in place of Proposition 
\ref{S(K,H,B) is nondegenerate iff}.
\end{proof}

\end{subsection}
%%%%%%%%%%%%%%%%%%%%%%%%%%%%%%%%%%%%%%%%%%%%%%%%%%%%%%%%%%%%%%%%%

\end{section}
%%%%%%%%%%%%%%%%%%%%%%%%%%%%%%%%%%%%%%%%%%%%%%%%%%%%%%%%%%%%%%%%%%

%%%%%%%%%%%%%%%%%%%%%%%%%%%%%%%%%%%%%%%%%%%%%%%%%%%%%%%%%%%%%%%%%%
\begin{section}
{A characterization of group-theoretical braided fusion categories}

Let us  recall notions of equivariantization and de-equivariantization
of fusion categories from \cite{Br, M2, Ki}.

Let $\C$ be a fusion category with an action of a finite 
group $G$. In this case one can define the fusion category $\C^G$ of 
$G$-equivariant objects in $\C$. 
An object of this category is an object $X$ of $\C$ equipped with 
an isomorphism $u_g: g(X)\to X$ for all $g\in G$, such that
$$
u_{gh}\circ \gamma_{g,h}=u_g\circ g(u_h),
$$
where $\gamma_{g,h}: g(h(X))\to gh(X)$ is the natural isomorphism
associated to the action. Morphisms and tensor
product of equivariant objects 
are defined in an obvious way. This category is called
the {\em $G$-equivariantization} of $\C$. 
%One has $\FPdim(\C^G)=|G|\FPdim(\C)$. 

There is a procedure opposite to equivariantization, called {\em de-equivariantization}.
In the context of modular categories it was introduced as a {\em modularization}
by Brugui\`{e}res and M\"uger.
Namely, let $\C$ be a fusion category  and let
$\E= \Rep(G) \subset \Z(\C)$  be a Tannakian subcategory that embeds into $\C$
via the forgetful functor $\Z(\C)\to \C$. Let $A=\mbox{Fun}(G)$ be the algebra of functions on
$G$. It is a commutative algebra in $Z(\C)$ and so
the category $\C_G$ of left $A$-modules in $\C$ is a
fusion category, called the {\em de-equivariantization} of $\C$ by $\E$. 
The  free module functor $\C \to \C_G: X \mapsto A\ot X$ is
a surjective tensor functor. 
%One has $\FPdim(\C_G) = \FPdim(\C)/ |G|$.

The above constructions are inverse to each other. In particular,
$\C_G$ admits a canonical action of $G$ such that  there is
a canonical equivalence $(\C_G)^G \cong \C$. 

\begin{remark}
\label{used in proof}
The following consequence of the above constructions will be used in the proof
of Theorem~\ref{interpretation} below.  Given a braided fusion category $\D$
and a Tannakian subcategory $\Rep(G) \subset \D$ we have a commutative algebra 
$A=\mbox{Fun}(G)$ in $\D$. The category $\D_G$ of $A$-modules in $\D$ is a fusion category
with an action of $G$ and $\D$ is equivalent to the $G$-equivariantization of $\D_G$.
\end{remark}

The next theorem characterizes group-theoretical braided fusion
categories as equivariantizations of pointed fusion categories.

\begin{theorem}
\label{interpretation}
Let $\D$ be a braided fusion category. Then $\D$ is group-theoretical if and only if
it contains a Tannakian subcategory $\E=\Rep(G)$ such that the corresponding 
de-equivariantization $\D_G$ is pointed. 
Equivalently, $\D$ is group-theoretical
if and only if it is an equivariantization of a pointed fusion category.
\end{theorem}
\begin{proof}
It follows from \cite[Theorem 3.5]{Nk} that an equivariantization of a pointed
fusion category is group-theoretical. Hence, if $\D_G$ is pointed then 
$\D \cong (\D_G)^G$ is group-theoretical.

Conversely, suppose $\D$ is group-theoretical. Then by Remark \ref{remark:braided}
and Theorem \ref{theorem:twisted-main}, there exist a group $G$,
a $3$-cocycle~$\omega$ on it, normal subgroups $K, H \subset G$, and 
$B: K \times H \to k^\times$
such that $\D =\mathcal{S}(K,\, H,\, B) \subset \Rep(D^\omega(G))$. 
The category $\D$ contains a Tannakian subcategory $\Rep(G/H)$ and 
is identified with a certain category of equivariant vector bundles on $K$.
Let $\tilde{\omega}$ denote the restriction of $\omega$ to $K$, let
$F: \D \to \Vec_K^{\tilde\omega}$ be the restriction of
the forgetful functor $\Rep(D^\omega(G)) \cong \Z(\Vec_G^\omega) \to \Vec_G^\omega$,
and let $I~:~\Vec_K^{\tilde\omega} \to \D$ be the left adjoint to $F$.
Then $A:= I(\be) = \mbox{Fun}(G/H)$ is a commutative algebra in $\Rep(G/H)\subset \D$.
Furthermore, for any $X\in \Vec_K^{\tilde\omega}$ the object $I(X)$ has a structure
of an $A$-module (see \cite[proof of Theorem 4.8]{DGNO} for details)  and the functor $X\mapsto I(X)$ 
is a tensor equivalence between $\Vec_K^{\tilde\omega}$ and the category of $A$-modules 
in $\D$. By Remark~\ref{used in proof}, $\D$ is equivalent to a $G/H$-equivariantization 
of the pointed category $\Vec_K^{\tilde\omega}$, so the proof is complete.
\end{proof}

\begin{remark}
\begin{enumerate}
\item[(i)] A different proof of Theorem~\ref{interpretation} will be given in an updated version of \cite{DGNO}.
The proof presented here uses the explicit description of group-theoretical
categories as subcategories of twisted group doubles.
\item[(ii)] By the work of Kirillov Jr.\ \cite{Ki} and M\"uger \cite{M2} the above de-equivari\-antization 
of $\D$ has a structure of a {\em braided $G$-crossed category} in the sense of Turaev \cite{T}.
By definition, a {\em braided $G$-crossed fusion category} is a fusion category $\C$
equipped with  an action $g\mapsto T_g$ of a group $G$ by tensor autoequivalences of  $\C$,
a (not necessarily faithful) grading $\C=\oplus_{g\in G}\,\C_g$,
and  a natural collection of isomorphisms 
\begin{equation*}
%\label{crosssed braiding}
c_{X,Y}: X\ot Y\simeq T_g(Y)\ot X, \qquad   X\in \C_g,\, g\in G\quad \mbox{ and } \quad  Y\in \C
\end{equation*}
satisfying certain  compatibility conditions.
% In \cite{DGNO} a different  proof of Theorem~\ref{interpretation} is given.
%The proof presented here uses the explicit description of group-theoretical
%categories as subcategories of twisted group doubles.
\item[(iii)] Recall that
a {\em crossed module} is a pair of groups  $(G,\, X)$ with $G$ acting on $X$ by automorphisms,
denoted $(g, x) \mapsto \lexp{g}{x}$, and a group homomorphism $\partial : X \to G$ 
satisfying
$$ 
\lexp{\partial(x)}{x'} = xx'x^{-1}, \qquad \text{for all }
x,x' \in X
$$
and
$$ 
\partial(\lexp{g}{x}) = g \partial(x) g^{-1}, \qquad \text{for all }
g \in G, x \in X.
$$
Let $\C$ be a pointed $G$-crossed braided category  and let $X$ be the group
of isomorphism classes of simple objects of $G$. Then $(G,\, X)$
is crossed module.  By Theorem \ref{interpretation} every group-theoretical braided 
fusion category gives rise to a crossed module. 

A construction of braided fusion categories from crossed modules was given by Bantay 
in \cite{Ba}.  But a general problem of classifying pointed $G$-crossed categories
in terms of group cohomology seems to be quite complicated, see \cite[4.9]{M4}.
\end{enumerate}
\end{remark}

\end{section}

%%%%%%%%%%%%%%%%%%%%%%%%%%%%%%%%%%%%%%%%%%%%%%%%%%%%%%%%%%%%%%%%%%%%%%%%%%%%%%%%%%%%%%%%%%%%%%%%%

%%%%%%%%%%%%%%%%%%%%%%%%%%%%%%%%%%%%%%%%%%%%%%%%%%%%%%%%%%%%%%%%%%%%%%%%%%%%%%%%%%%%%%%%%%%%%%%%%
\end{document}